\documentclass[10pt,amsfonts, epsfig]{amsart}
\usepackage{amsmath, amscd, amssymb}
\usepackage{graphpap, color}
\usepackage{mathrsfs}
\usepackage{pstricks}
\usepackage{color}
\usepackage{cancel}
\usepackage[mathscr]{eucal}
\usepackage{pstricks}
\usepackage{color}
\usepackage{cancel}
\usepackage{verbatim}
\usepackage{latexsym}
\usepackage[all]{xy}
\def\Sbul{\mathrm{Sym}}

\usepackage{upgreek}

 \usepackage{wasysym}
\usepackage{ esint }

\def\bC{{\mathbf C}}

\def\bW{{\mathbf W}}

\def\fu{\mathfrak u}

\def\ff{{\mathfrak f}}
\def\gg{{\mathfrak g}}

\def\fu{\mathfrak{u}}

\numberwithin{equation}{section}

\def\loc{_1^\circ}

\def\sO{{\mathscr O}}
\def\sC{{\mathscr C}}

\def\sL{{\mathscr L}}

\def\sO{\mathscr{O}}

\def\sE{\mathscr{E}}
\def\sF{\mathscr{F}}

\newcommand{\CC}{\mathbb{C}}

\newcommand{\EE}{\mathbb{E}}
\newcommand{\LL}{L\bul} 
\newcommand{\tLL}{L^{\bullet \geq -1}}
\newcommand{\PP}{\mathbb{P}}
\newcommand{\QQ}{\mathbb{Q}}

\newcommand{\ZZ}{\mathbb{Z}}

\newcommand{\FF}{\mathbb{F}}

\def\sK{{\mathscr K}}

\newcommand{\bm}{\mathbf{m}}
\newcommand{\bn}{\mathbf{n}}

\newcommand{\cal}{\mathcal}
\def\ee{\mathfrak e}

\def\cC{{\cal C}}

\def\cE{{\cal E}}

\def\cK{{\cal K}}
\def\cL{{\cal L}}
\def\cM{{\cal M}}

\def\cO{{\cal O}}

\def\cS{{\cal S}}
\def\cX{X} 
\def\cY{{\cal Y}}
\def\M{M}

\def\bm{{\mathbf m}}

\def\fV{\mathfrak{V}}

\def\ff{\mathfrak{f}}

\def\ft{\mathfrak{t}}
\def\fu{\mathfrak{u}}
\def\fv{\mathfrak{v}}

\def\v1{{\vec{1}}}

\def\mapright#1{\,\smash{\mathop{\lra}\limits^{#1}}\,}
\def\toright#1{\,\smash{\mathop{\to}\limits^{#1}}\,}

\def\dual{^{\vee}}

\def\sta{^\ast}

\def\virt{^{\mathrm{vir}}}
\def\upmo{^{-1}}
\def\sta{^{\ast}}

\def\sta{^*}

\def\lra{\longrightarrow}

\def\lsta{_{\ast}}

\newcommand{\Si}{\Sigma}

\newcommand{\lam}{\lambda}
\newcommand{\si}{\sigma}

\def\begeq{\begin{equation}}
\def\endeq{\end{equation}}
\def\and{\quad{\rm and}\quad}
\def\bl{\bigl(}
\def\br{\bigr)}
\def\defeq{:=}

\def\sub{\subset}
\def\Ao{{\mathbb A}^{\!1}}

\def\and{\quad\text{and}\quad}

  \DeclareMathOperator{\Hom}{Hom}
 \DeclareMathOperator{\Aut}{Aut}

 \DeclareMathOperator{\rank}{rank}
\DeclareMathOperator{\spec}{Spec}

\newtheorem{prop}{Proposition}[section]
\newtheorem{theo}[prop]{Theorem}
\newtheorem{lemm}[prop]{Lemma}
\newtheorem{coro}[prop]{Corollary}

\newtheorem{defi}[prop]{Definition}

\newtheorem{defi-prop}[prop]{Definition-Proposition}
\newtheorem{defi-theo}[prop]{Definition-Theorem}

\def\fu{\mathfrak{u}}

\def\dbar{\overline{\partial}}

\def\Ob{\cO b}

\def\loc{_{\mathrm{loc}}}

\def\bul{^\bullet}

\def\sta{^\ast}

\def\oD{{\bar{D}}}

\def\Gm{\mathbb G_{\mathrm m}}

\def\sO{{\mathscr O}}

\def\beq{\begin{equation}}
\def\eeq{\end{equation}}

\def\vsp{\vskip5pt}
\def\Pf{{\PP^4}}

\def\bee{\begin{equation}}
\def\eeq{\end{equation}}

\def\sC{{\mathscr C}}

\def\fV{{\mathfrak V}}

\def\fW{{W}}

\def\bd{\delta} 

\def\ti{\tilde}

\def\fg{{\mathfrak g}}
\def\barM{{\overline{M}}}

\def\Vb{\mathrm{Vb}}

\def\ulog{^{log}}

\def\mapright#1{\,\smash{\mathop{\lra}\limits^{#1}}\,}

\def\Gm{\mathbb G_{\mathrm m}}

 \def\R{\mathbf R}
\def\Rpi{\mathbf R\pi}

\title{Witten's top Chern class via cosection localization}

\author{H-L. Chang, J. Li and W-P.Li}

\date{}

\begin{document}
\maketitle


\begin{abstract}  For a Landau Ginzburg space $([\CC^n/G],W)$,  we construct the Witten's top Chern classes
as algebraic cycles 
via cosection localized virtual cycles in case all sectors are narrow. We verify all axioms of
such classes.  We derive an explicit formula of such classes in the free case. We
prove that this construction is equivalent to the prior constructions of Polishchuk-Vaintrob, of Chiodo and of
Fan-Jarvis-Ruan.   
\end{abstract}

 

\section{Introduction}

In this paper, we construct and study Witten's top Chern class of the moduli of spin curves associated
to a Landau Ginzburg space using cosection localized virtual classes.

A Landau Ginzburg space (in short LG space) in this paper is a pair $([\CC^n/G],W)$ of a finite subgroup $G\le GL_n(\CC)$ and a non-degenerate quasi-homogeneous  $G$-invariant  polynomial $W$ in $n$-variables (see Definition \ref{WW}).
Given such an LG space, one forms the moduli stack of smooth $G$-spin curves $M_{g,\ell}(G)$, which when
$G=\Aut(W)$, takes the form
\beq\label{Wlog}
M_{g,\ell}(G)=\{[\cC,\cL_1,\cdots,\cL_n]\mid \cC \text{ stable}, W_a(\cL_1,\cdots,\cL_n)\cong \omega_{\cC}\ulog\}.
\eeq 
Here $\cC$ is a smooth $\ell$-pointed twisted (orbifold) curve, $\cL_j$'s are invertible sheaves on $\cC$,
and $W_a$'s are the monomials of $W$ (see 
details in \S2).
In \cite{Wi}, 
Witten demonstrated how to construct a ``topological gravity coupled with matter" using solutions to
his equation (i.e. Witten equation), which takes the form 
$$\dbar s+ (k+1) \bar s^k=0,\quad s\in C^\infty(\cC, \cL),
$$
in case of $([\CC/\ZZ_{k+1}], W=x^{k+1})$ and for $[\cC,\cL_j]\in M_{g,n}(G)$.
He conjectured that the 
partition functions of such $A_k$  singularities, and also other singularities of  $DE$ type, satisfy ADE integrable hierarchies.


The mathematical theory of Witten's ``topological gravity coupled with matter" has satisfactorily been worked out. 
The proper moduli of nodal spin curves has been worked out by   Abramovich and  Jarvis (\cite{Ja1,Ja2,AJ}).
The ``Witten's top Chern class" has been constructed by  Polishchuk-Vaintrob (\cite{PV1}), alternatively by Chiodo (\cite{Ch2,Ch3}) via $K$-theory, and by Mochizuki (\cite{Mo}) following Witten's approach.

The case for a general LG space has been worked out later by Fan-Jarvis-Ruan (cf. \cite{FJR1,FJR2}). 
Their construction is analytic in nature, and uses the Witten equations for $W$
\beq\label{Wn}\dbar s_j+\overline {\partial_j W}(s_1,\cdots,s_n)=0,  \quad s_j \in C^\infty(\cC,\cL_j),
\eeq
to construct the Witten's top Chern class of $M_{g,\ell}(G)$; they also proved all expected properties of
such classes. In line of that the GW theory of a smooth variety is a virtual counting of maps, this theory,
now commonly referred to as FJRW-theory, is an (enumeration) theory for the singularity $(W=0)/G$. 
We add that since the domain curves $\cC$ in $[\cC,\cL_j]\in M_{g,\ell}(G)$ are pointed twisted curves,
each $\cL_j$ associats to a representation (sector) of the automorphism group of a marked point of $\cC$. We use $\gamma$ to denote
this collection of representations. Those $\gamma$ such that all factors of the representation at all marked points
are non-trivial are called narrow (``Nevau-Schwarz"). The FJRW-theory treated all situations, including but
not restricted to the narrow case.
Witten's ADE integrable conjecture was solved by Faber-Shadrin-Zvonkine  for $A_n$ case (\cite{FSZ}), and by Fan-Jarvis-Ruan for DE case
(\cite{FJR2}).

In this paper, using cosection localized virtual cycles,  we construct  the Witten's top Chern class for an LG space $([\CC^n/G],W)$
in case all sectors are narrow; we also verify all expected properties of the virtual class using this construction.
Our construction is an algebraic analogue of Witten's argument using his equation. This allows us to prove 
that our construction yields the same classes as that of
Polishchuk-Vaintrob, of Chiodo and of Fan-Jarvis-Ruan. We comment that the equivalence of that of
Polishchuk-Vaintrob and of Chiodo is known; the equivalence of that of Polishchuk-Vaintrob and of Fan-Jarvis-Ruan
is new.
\vsp

We define the Witten's top Chern class to be the cosection localized virtual class of the moduli of $G$-spin curves with fields.
Let $([\CC^n/G],W)$ be an LG space.
Given integers $g$, $\ell$, and a collection of representations $\gamma$,
 denote $\barM_{g,\gamma}(G)$ to be the moduli stack of $G$-spin $\ell$-pointed genus $g$ twisted nodal curves banded by
$\gamma$, which parameterizes $[\cC,\cL_1,\cdots,\cL_n]$ so that $\cC$ are stable 
$\ell$-pointed genus $g$ twisted nodal curves and $\cL_j$'s are  invertible sheaves on $\cC$ such that, in addition to the
constraint in \eqref{Wlog} (when $G=\Aut(W)$)
the representations of $\cL_j$ restricted to the 
marked points of $\cC$ are given by the collection $\gamma$.
Following the work of the first two named authors (\cite{CL}), we form the moduli of $G$-spin curves with fields:
$$\barM_{g,\gamma}(G)^p=\{[\cC,\cL_j,\rho_j]_{i=1}^n\mid [\cC,\cL_j]\in \barM_{g,\gamma}(G),
\, \rho_j\in \Gamma(\cL_j)\}.
$$
It is a DM-stack, and has a perfect obstruction theory 
relative to $\barM_{g,\gamma}(G)$.
The forgetful morphism
\beq\label{11}
\barM_{g,\gamma}(G)^p \lra \barM_{g,\gamma}(G)
\eeq
has linear fibers, and has a zero-section by setting all $\rho_j=0$.

Using that $G\sub \Aut(W)$, and that $\gamma$ is narrow,
we use the polynomial $W$ to construct a cosection (homomorphism)
of the obstruction sheaf of $\barM_{g,\gamma}(G)^p$:
\beq\label{si-int}
\sigma: \Ob_{\barM_{g,\gamma}(G)^p}\lra \sO_{\barM_{g,\gamma}(G)^p}.
\eeq
We prove that the non-surjective loci of $\sigma$ is contained in the zero section of
\eqref{11}, 
which is $\barM_{g,\gamma}(G)$ and is proper. Applying cosection localized virtual class of
Kiem and the second named author (\cite{KL}), we obtain a cosection localized virtual class of $\barM_{g,\gamma}(G)^p$,
which we denote by 
\beq\label{W-class}
[\barM_{g,\gamma}(G)^p]\virt \in A\lsta (\barM_{g,\gamma}(G)).
\eeq

\begin{defi-theo}
Let $([\CC^n/G],W)$ be an LG space; let $g$, $\ell$ be non-negative integers and let $\gamma$ be a
collection of representations. Suppose $\gamma$ is narrow, 
we define the (cosection localized) Witten's top Chern class $\barM_{g,\gamma}(G)$ of
$([\CC^n/G],W)$ to be the cosection localized virtual class $[\barM_{g,\gamma}(G)^p]\virt$.
\end{defi-theo}

We verify the expected properties (axioms) of Witten's top Chern classes in \S 4.
And we prove the following comparison theorem.

\begin{theo}
The Witten's top Chern class $[\barM_{g,\gamma}(G)^p]\virt$ via cosection localization coincides with the Witten's top Chern class constructed by Polishchuk-Vaintrob when their construction applies;
its associated homology class coincides with the analytic 
Witten's top Chern class constructed by Fan-Jarvis-Ruan (in the narrow case).
\end{theo}

We comment on the proof of the comparison theorem. 
Looking closer at the Witten's equation \eqref{Wn}, and realizing that the term $\dbar s_j$
gives the obstruction class to extending a holomorphic sections of $\cL_j$, the equation \eqref{Wn}
in effect gives a (differentiable) section of the obstruction sheaf of the moduli of spin curves with fields. Witten's demonstration of his top Chern class could be viewed as using the homology class generated by
the solution space of transverse perturbation of the equation \eqref{Wn}.

Working algebraically,
we substitute the complex conjugation used in \eqref{Wn} by the Serre duality, and thus transform the equation
\eqref{Wn} to the cosection \eqref{si-int}. As the cosection has proper non-surjective
loci, the cosection localized Gysin map gives us a virtual
cycle of $\barM_{g,\gamma}(G)^p$ supported in $\barM_{g,\gamma}(G)$. Using topological nature of the
cosection localized virtual class, we show that our construction yields the same class
as the class constructed by
FJRW-theory, when pushed to the ordinary homology group. 

To the algebro-geometric construction of Witten's top Chern class for $(\CC, x^{k+1})$ by Polishchuk-Vaintrob,
it relies on resolving the universal family of the moduli of spin curves
and define Witten's top Chern class as certain combination of Chern classes of complexes derived from
the resolution. 
We show in \S 5 that because the relative obstruction theory of $\barM_{g,\gamma}(G)^p$ to
$\barM_{g,\gamma}(G)$ is linear, the cosection localized virtual cycle can be expressed
in terms of localized Chern classes of certain complexes. 
This leads to a proof of the comparison theorem.

This paper is organized as follows. In \S 2, we recall the notion of $G$-spin curves with fields.
The Witten's top Chern class will be constructed in \S 3. 
In \S 4, we verify the axioms of such class, and derive a closed formula of the Witten's top Chern class in 
the free case. The comparison theorem will be proved in \S 5.

{\sl Acknowledgement}.  The authors thank T. Jarvis, B. Fantechi, A. Chiodo, H-J. Fan, Y-F. Shen and S. Li 
for helpful discussions and explanations of their results. 
The authors thank Y-B. Ruan for explaining various aspects of FJRW invariants. 
The first author is partially supported by Hong Kong GRF grant 600711. The second author is partially supported by NSF grants 
NSF-1104553 and FRG-1159156. The third author is partially supported by Hong Kong GRF grant 602512.

\section{generalized spin curves with fields}
We work with the field of complex numbers and let $\Gm=GL(\CC)$. In this section we recall the moduli of $G$-spin curves with fields for a finite subgroup
$G\le\Gm^n$.  Our treatment follows that of \cite{FJR2,PV2}, 
except that we use the term ``$G$-spin curve" instead of ``$W$-curve" in \cite{FJR2} or ``$\Gamma$-spin curve" in \cite{PV2}, as the moduli only depends on the group $G$.  

\subsection{LG spaces}
We fix an integer $d>1$ and a primitive $n$-tuple 
$$\bd=(\delta_1,\cdots, \delta_n)\in \ZZ_+^{n}.
$$
Let $\zeta_d=\exp(\frac{2\pi\sqrt{-1}}{d})$ be the standard generator of the group  $\mu_d\le \Gm$ of the
$d$-th roots of unity. We define 
$$\jmath_\bd=(\zeta_d^{\delta_1},\cdots, \zeta_d^{\delta_n})\in (\mu_d)^n \and  \text{the subgroup}\ \ \langle \jmath_\bd\rangle\le  (\mu_d)^{n}.
$$

\begin{defi}\label{G} A finite subgroup $G\le \Gm^n$ containing $\langle\jmath_\bd\rangle$ is called a $(d,\bd)$-group.
\end{defi}


\begin{defi}\label{WW}
We say a Laurent polynomial $W$ in variables $(x_1,\cdots, x_n)$ quasi-homogeneous 
of weight $(d,\bd)$ if 
\beq\label{d}
W(\lambda^{\delta_1}x_1, \ldots, \lambda^{\delta_n}x_n)=\lambda^dW(x_1, \ldots, x_n).
\eeq
When $W$ is a polynomial, we say it is non-degenerate if
$W$ has a single critical point at the origin, and $W$ does not contain  terms $x_ix_j$, $i\ne j$.
\end{defi}

Let $W$ be a weight $(d,\bd)$ quasi-homogeneous Laurent polynomial. 
We write    $W=\sum_{a=1}^m \alpha_aW_a$,   $\alpha_a\ne 0$,  as a sum of Laurent monomials $W_a$. 
These monomials define a group homomorphism
 \beq\label{homo}
\tau_W\colon \Gm^n\lra \Gm^m,\quad
x=(x_1, \ldots, x_n) \mapsto    (W_1(x), \ldots, W_m(x)).
\eeq 
We define 
$$\Aut(W):=\ker(\tau_W)\le \Gm^n.
$$
As $W$ is of weight $(d,\bd)$, 
$\jmath_\bd\in \Aut(W)$. When $\Aut(W)$ is finite, it is a $(d,\bd)$-group.
Also note that if $W$ is a nondegenerate polynomial as in Definition \ref{WW}, then $\Aut(W)$ is finite (\cite[Lem. 2.1.8]{FJR2}).

\begin{defi}\label{def-LG}
We say $([\CC^n/G],W)$ is an LG space if there is a pair $(d,\delta)$ so that $G$ is a $(d,\delta)$-group, and
$W$ is a non-degenerate quasi-homogeneous polynomial of weight $(d,\delta)$ such that $G\le \Aut(W)$.
\end{defi}



\vsp

We remark that given an LG space $([\CC^n/G],W)$, the orbifold 
$[\CC^n/G]$ with the superpotential $W$ is called an ``affine Landau Ginzburg model" because it admits a global affine chart 
$\CC^n\to [\CC^n/G]$.  One can work with `Hybrid LG model", LG spaces such as 
$(K_{\Pf},\bW)$ given in the Guffin-Sharpe-Witten model, whose mathematical construction for all genus (after A-twist) is carried out in \cite{CL}.


 

\subsection{Twisted curves}

We follow the notations and definitions in \cite{A-V, A-G-V}.

\begin{defi}\label{twisted-curves}
An $\ell$-pointed twisted  nodal curve over a scheme $S$ is a datum
\begin{eqnarray*}
(\Sigma_i^{\sC}\subset \sC\to C\to S)
\end{eqnarray*}
where

\begin{enumerate}
\item[({\bf a})] $\sC$ is a proper DM stack,   and \'etale locally a nodal curve over $S$.

\item[({\bf b})] $\Sigma_i^{\sC}$, for $1\le i\le \ell$,  are disjoint closed substacks of $\sC$ in the smooth locus of $\sC\to S$, and
$\Sigma_i^{\sC}\to S$ are \'etale  gerbes banded by $\mu_{r_i}$ for some $r_i$.

\item[({\bf c})]  The morphism $\pi\colon \sC\to C$ makes $C$ the coarse moduli of $\sC$.

\item[({\bf d})] Near stacky-nodes, $\sC$ is  balanced.

\item[({\bf e})]  $\sC\to C$ is an isomorphism over $\sC_{\rm ord}$, where $\sC_{\rm ord}$
is the complement of markings $\Sigma_i^\sC$ and the
stacky-singular locus of the projection $\sC\to S$.
\end{enumerate}
\end{defi}

Here are some  comments on this definition.
A stacky-node is a node that locally is not a scheme; that $\sC$ is {\it balanced} near nodes means that,
over a strictly Henselian local ring $R$, the strict henselization $\sC^{sh}$ is isomorphic to the stack $[U/{\bf \mu}_r]$,
where $U={\rm Spec}(R[x, y]/(xy-t))^{sh}$
for some $t\in R$ and $\zeta\in{ \mu}_r$ acts via $\zeta\cdot (x, y)=(\zeta x, \zeta^{-1}y)$.
Near the marking $\Sigma_i^{\sC}$, $\sC^{sh}$  is   isomorphic to $[U/\mu_r]$,
where  $U={\rm Spec}R[z]^{sh}$ and 
$\zeta\in \mu_r$ acts on $U$ via $z\to \zeta z$ for $z$  a local
coordinate on $U$ defining the marking. The automorphism group of $\sC$ at $\Sigma_i^\sC$ is canonically isomorphic to $\mu_r$.
Denote by $\Sigma_i^C$ the coarse moduli space of $\Sigma_i^{\sC}$, then the inclusion $\Sigma_i^C\sub C$ makes $(C, \Sigma_i^C$)
a nodal pointed curve.
We denote $\Sigma^\sC=\coprod_i \Sigma^\sC_i$ and $\Sigma^C=\coprod_i \Sigma^C_i$. 
We define the log-dualizing sheaves to be
$\omega_{\sC/S}^{\text{log}}=\omega_{\sC/S}(\Sigma^{\sC})$ and $\omega_{C/S}^{\text{log}}=\omega_{C/S}(\Sigma^{C})$, 
where $\omega_{\sC/S}$ and $\omega_{C/S}$ are dualizing sheaves of $\sC/S$ and $C/S$, respectively. 
Note $\omega_{\sC/S}^{\text{log}}=\pi^*\omega_{C/S}^{\text{log}}$ (\cite[Sec 1.3]{AJ}).

When there is no confusion, we will use $(\sC, \Sigma_i^{\sC})$, or simply $\sC$, to denote
$(\Sigma_i^{\sC}\subset \sC\to C\to S)$. 
We denote by $\mathfrak M_{g, \ell}^{\rm tw}$ the moduli stack of genus $g$ $\ell$-pointed twisted nodal  curves.

\subsection{$G$-spin curves}

For any   $(d,\bd)$-group $G$, we introduce
$$\Lambda_G=\{ \bm \mid  \bm \text{ are $G$-invariant Laurent monomials in } (x_1,\cdots, x_n) \}.
$$
 Since $\langle \jmath_\bd\rangle\sub G$, every $\bm\in \Lambda_G$  is of weight $ (w(\bm)\cdot d,\bd)$, where 
$$w(\bm)=d\upmo\cdot  \deg \bm(t^{\delta_1},\cdots,t^{\delta_n})\in \ZZ.
$$

\begin{lemm}
Via standard multiplication, $\Lambda_G$ is a free abelian group of rank $n$.
\end{lemm}
 \begin{proof} Every Laurent monomial can be represented as $x_1^{c_1}\cdots x_n^{c_n}$ with integers $c_1,\cdots,c_n$. In this way the multiplicative group
of Laurent polynomials is isomorphic to the additive group $\ZZ^n$. Hence $\Lambda_G$ is a subgroup of $\ZZ^n$. By definition
the $|G|$-th power of every element in $\ZZ^n/\Lambda_G$ vanishes. This proves the Lemma.
\end{proof}

One checks that $G=\{x\in\Gm^n \mid \bm(x)=1\, \text{for all }\bm\in\Lambda_G\}$. 
We pick $n$ generators $\bm_1,\cdots,\bm_n$ of $\Lambda_G$.

\begin{defi}\label{W-curves}
An $\ell$-pointed $G$-spin curve over a scheme $S$ is $(\sC,\Sigma_i^\sC, \sL_j,\varphi_k)$, consisting 
of an $\ell$-pointed twisted curve $(\sC,\Sigma^\sC_i)$ over $S$, $n$-invertible sheaves $\sL_1,\cdots, \sL_n$  
on $\sC$, and isomorphisms 
\beq\label{Ow}
\varphi_k\colon
\bm_k(\sL_1, \ldots, \sL_n)\mapright{\cong}(\omega_{\sC}^{\text{log}})^{w(\bm_k)}  ,\quad k=1,\cdots,n.  
\eeq
 An arrow between two $G$-spin curves $(\sC, \sL_j,\varphi_k)$ and $(\sC', \sL_j',\varphi_k')$ over $S$
consists of  $(\sigma, \eta_j)$, where $\sigma: \sC\to\sC'$ is an $S$-isomorphism of pointed
twisted curves, and $\eta_j: \sigma\sta\sL_j'\to \sL_j$ are isomorphisms that commute with the
isomorphisms $\varphi_k$ and $\varphi_k'$. 
\end{defi}

We comment that the definition of $G$-spin curves does not depend on the choice of the generators 
$\{\bm_k\}$ (cf. \cite[Prop 2.1.12]{FJR2}), once the type $(d,\bd)$ is specified.

\vsp
A $G$-spin curve $(\sC,\Sigma_i^\sC, \sL_j,\varphi_k)$ (over $\CC$) has monodromy representations along marked sections
and nodes. By definition, the automorphisms of each $\Sigma_i^\sC$ is a cyclic group, say $\mu_{r_i}$, which
acts on $\oplus_{j=1}^n \sL_j|_{\Sigma^\sC_i}$, and thus defines a homomorphism $\mu_{r_i}\to \Gm^n$.
Because of (\ref{Ow}) and that $\{\bm_k\}_{k=1}^n$ generates $\Lambda_G$, 
this homomorphism factors through a homorphism
\begin{eqnarray}\label{representable}
\gamma_i:\mu_{r_i}\longrightarrow G\le \Gm^n. 
\end{eqnarray}
 We call $\gamma_i$ the monodromy representation along $\Sigma_i^\sC$.

Similarly, for $q$ a node of $\sC$, we let $\hat\sC_{q+}$ and $\hat\sC_{q-}$ be the two
branches of the formal completion of $\sC$ along $q$, of the form $[\hat U/\mu_{r_q}]$, where
$\hat U=\spec\CC[\![x,y]\!]/(xy)$ and $\mu_{r_q}$ acts on $\hat U$ via $(x,y)^\zeta=(\zeta x,\zeta\upmo y)$,
and $\hat \sC_{q+}$ (resp. $\hat \sC_{q-}$) is $(y=0)\sub \hat U$ (resp. $(x=0)\sub \hat U$).
We let
$$\gamma_{q\pm}: \mu_{r_q}\lra G\le\Gm^n
$$
be the monodromy representation of $\oplus_{j=1}^n \sL_j\otimes_{\sO_\sC}\sO_{\hat \sC_{q\pm}}$ along
$[q/\mu_{r_q}]\sub \hat\sC_{q\pm}$. .

Denoting by $\gamma_{q+}\cdot\gamma_{q-}$ the composition of $(\gamma_{q+},\gamma_{q-}):\mu_{r_q}\to G\times G$
with the multiplication $G\times G\to G$, then by the balanced condition on nodes, we have
$\gamma_{q+}\cdot\gamma_{q-}=1$, the trivial homomorphism. We call $\gamma_{q+}$ the
mododromy representation of the node $q$, after a choice of $\hat\sC_{q+}$.  

\begin{defi}\label{def-2.7}
We say a $G$-spin curve $(\sC,\Sigma_i^\sC, \sL_j,\varphi_k)$ is stable if its coarse moduli space 
$(C, \Sigma^C_i)$ is a stable pointed curve, and if the monodromy representations of marked sections and nodes
are injective (representable).
\end{defi}

In this paper, given $\ell$, we use $\gamma=(\gamma_i)_{i=1}^\ell$ to denote a collection of injective homomorphisms
$\gamma_i: \mu_{r_i}\to G$ for some choices of $r_i\in\ZZ_+$.
Thus every stable $G$-spin curve with $\ell$ marked sections  will associate one such $\gamma=(\gamma_i)_{i=1}^\ell$ via 
monodromy representations. If $\gamma=(\gamma_i)^\ell_{i=1}$ is associated to some stable genus $g$ $G$-spin curves, we
call such $\gamma$ $g$-admissible.


\begin{lemm}[{\cite[Prop 2.2.8]{FJR2}}] \label{lem2.8}
Given non-negative integers $g$, a collection of
faithful representations $\gamma=(\gamma_1,\cdots, \gamma_\ell)$ is $g$-admissible if and only if, 
writing $\gamma_i$ in the form
$\mu_{r_i}\ni e^{2\pi \sqrt{-1}/r_i} \mapsto 
(e^{ 2\pi \sqrt{-1}  \Theta^i_1},\cdots , e^{2\pi \sqrt{-1} \Theta^i_n})$, $\Theta_j^i\in [0,1)$, the following
identity hold:
\beq\label{del}
\delta_j(2g-2+\ell)/d-\sum_{i=1}^n\Theta_j^i\in \ZZ, \quad j=1,\cdots, n.
\eeq  
\end{lemm}

\begin{defi}\label{gamma} Given $g$ and let $\gamma$ be a $g$-admissible collection of representations.
We say a $G$-spin curve $(\sC,\Sigma_i^\sC, \sL_j,\varphi_k)$ is banded by $\gamma$ if $\gamma$ is identical
to the collection of representations $\Aut(\Sigma_i^\sC)\to \Aut(\oplus_{j=1}^n\sL_j|_{\Sigma_i^\sC})$.
\end{defi}


Given a $g$-admissible $\gamma$,
it is routine to define the notion of families of genus $g$ $\gamma$-banded $G$-spin curves; 
define arrows between two such families,  and define pullbacks. Accordingly, we define 
$\barM_{g, \gamma}(G)$ 
to be the groupoid of families of stable genus $g$, $\gamma$-banded $G$-curves. 
We define 
$$\barM_{g, \ell}(G)\defeq \coprod_\gamma  \barM_{g, \gamma}(G),
$$
where $\gamma$ runs through all possible $g$-admissible $\gamma$. 
\vsp

The stack $\barM_{g, \ell}(G)$ is a smooth proper DM-stack with projective coarse moduli;
the forgetful morphism from $\barM_{g, \ell}(G)$ to the moduli space $ \barM_{g,\ell}$ of $\ell$-pointed stable curves is quasi-finite  (cf. \cite{FJR2} and \cite[Prop 3.2.6]{PV2}). Thus
$\barM_{g, \gamma}(G)$ is a smooth proper DM-stack.

When $G\le G'$, then $\Lambda_{G'}\le\Lambda_G$ and the generators $\bm'_i$ of 
$\Lambda_{G'}$ can be expressed as products of the generators $\bm_i^{\pm 1}$ of $\Lambda_G$. 
This shows that the universal family of $\barM_{g,\gamma}(G)$ induces a  morphism
$\barM_{g,\gamma}(G)\to \barM_{g,\gamma}(G')$, independent of the choices involved. The induced morphism
between their coarse modulis is both open and closed embedding.

 

\vsp
We now suppose $G$ is part of an LG space $([\CC^n/G],W)$. Write $W=\sum_{a=1}^m \alpha_aW_a$,
where $\alpha_a\ne 0$; then $W_a\in \Lambda_G$.
Let $\bm_1,\cdots,\bm_n$ be the chosen generators of $\Lambda_G$. 
Then there are Laurent monomials $\bn_a$ so that $W_a=\bn_a(\bm_1,\cdots,\bm_n)$. 
Consequently, the isomorphisms
$\varphi_1,\cdots,\varphi_n$ in \eqref{Ow} induce
\beq\label{Oow}
\upvarphi_a\defeq \bn_a(\varphi_1,\cdots,\varphi_n)\colon W_a(\sL_1,\cdots, \sL_n)\mapright{\cong}
\omega_{\sC}\ulog,\quad 1\le a\le m.
\eeq
   
We next work with the universal family of $\barM_{g, \gamma}(G)$. Let
$\pi: \cC\to \barM_{g, \gamma}(G)$, invertible sheaves $\cL_1,\cdots,\cL_n$ over $\cC$,  and $n$ isomorphisms as in \eqref{Ow} be part of the universal family of $\barM_{g, \gamma}(G)$. Following the discussion preceding \eqref{Oow}, we obtain $m$ induced isomorphisms
\beq\label{uni}
\Phi_a: W_a(\cL_1,\cdots,\cL_n)\mapright{\cong} \omega^{log}_{\cC/\barM_{g, \gamma}(G)},\quad 1\le a\le m.
\eeq



%

%

\subsection{Moduli of $G$-spin curves with fields}
We begin with its definition.

\begin{defi}\label{w-curve-field} A stable $\gamma$-banded $G$-spin curve with fields
consists of a stable $\gamma$-banded $G$-spin curve 
$(\sC,\Sigma_i^\sC, \sL_j,\varphi_k) \in \barM_{g,\gamma}(G)$
together with $n$ sections
$\rho_j\in \Gamma(\sC, \sL_j)$, $j=1,\cdots, n$. 
\end{defi}

It is routine to form the stack of families of stable genus $g$, $\gamma$-banded $G$-spin curves with fields. 
We denote this stack by $\barM_{g,\gamma}(G)^p$. 

\begin{theo}\label{thm2.8}
For the data $g$, $G$ and $\gamma$ given, the stack
$\barM_{g,\gamma}(G)^p$ is a separated DM stack of finite type.
\end{theo}

\begin{proof}
For notational simplicity, we abbreviate $\M=\barM_{g,\gamma}(G)$ and denote by
$\pi_M:\cC_M\to M$ with $\cL_{M,1},\cdots ,\cL_{M,n}$ the universal invertible sheaves of $M$. Let
$\cE_M=\cL_{M,1}\oplus\ldots \oplus \cL_{M,n}$. Following \cite{CL}, we denote by
$C(\pi_{M\ast}\cE_M)(S)$ the groupoid of $(f,\rho_S)$, where $f: S\to M$ and 
$\rho_S\in \Gamma(\cC_M\times_M S, f\sta\cE_M)$.
With obviously defined arrows among elements in $C(\pi_{M\ast}\cE_M)(S)$ and pullback 
$C(\pi_{M\ast}\cE_M)(S)\to C(\pi_{M\ast}\cE_M)(S')$
for $S'\to S$, they form a stack $C(\pi_{M\ast}\cE_M)$. It is direct to check (cf. \cite{CL}) that 
$C(\pi_{M\ast}\cE_M)$ is canonically isomorphic to
$\barM_{g,\gamma}(G)^p$, and is a DM stack, quasi-projective over $M$. 

Alternatively, using Grothendieck duality for DM stacks (cf.  \cite{Ni}), the proof of \cite[Prop. 2.2]{CL} shows that
$$C(\pi_{M\ast}\cE_M)=\spec_M  \Sbul R^1 \pi_{M\ast} (\cE_M\dual\otimes\omega_{\cC_M/M}).
$$
This proves the Theorem.
\end{proof}

We used the construction $C(\pi_{M\ast}\cE_M)$ because later on  we will  quote the construction of
the obstruction theory worked out in \cite{CL}.

\section{The relative obstruction theory and cosections}
In the next two sections, we fix an LG space $([\CC^n/G],W)$ of 
weight $(d,\bd)$.  We fix $g$ and $\ell$, and also a $\gamma=(\gamma_i)_{i=1}^\ell$,
where each $\gamma_i: \mu_{r_i}\to G$ is injective.

\begin{defi}\label{NS}
We say that $\gamma_i:\mu_{r_i}\to G$ is narrow (Neveu-Schwarz)
if the composite of $\gamma_i: \mu_{r_i}\to G\sub  \Gm^n$ with any projection to its factor $\Gm^n\to\Gm$ is non-trivial.
We say $\gamma=(\gamma_i)_{i=1}^\ell$ is narrow if every $\gamma_i$ is narrow.
\end{defi}

Like before, we abbreviate 
$$M= \barM_{g,\gamma}(G)\and \cX=\barM_{g,\gamma}(G)^p.
$$

\subsection{The perfect obstruction theory}

Let 
\beq\label{uni1}
\bl \Sigma_{\cX,i}
\sub\cC_\cX\mapright{\pi_\cX}\cX,\,  \cL_{\cX,j}, \, \varphi_k,\, 
 \rho_{\cX,j}\br 
\eeq
be the universal family
of $\cX$. By \cite[Prop 2.5]{CL}, $\cX= C(\pi_{M\ast}\cE_M) $  relative to 
$M$ has a tautological perfect obstruction theory, (letting $\cE_\cX:=\oplus \cL_{\cX,j}$,)
\beq\label{XY}\phi_{\cX/M}: (L\bul_{\cX/M})\dual \lra E\bul_{\cX/M}\defeq  \R \pi_{\cX\ast}\cE_\cX.
\eeq

\vsp
A consequence of this description of the perfect obstruction theory is a formula of its virtual dimension. 
Let $\xi=(C, \Sigma_i, L_j,\varphi_k, \rho_j)$ be a closed point of $\cX$. 
Following the notation of Lemma \ref{lem2.8}, 
the Riemann-Roch theorem \cite[Thm.7.2.1]{A-G-V} calculates the  virtual dimension of $\cX/ M$  to be
$$\text{dim}\, H^0(E\bul_{\cX/M}|_\xi)-\text{dim}\, H^1(E\bul_{\cX/M}|_\xi)
=n(1-g)+\sum_j \deg L_j  - \sum_{i,j} \Theta^i_j.$$
Combined with $\dim M=3g-3+\ell$, we obtain
\beq\label{vd}
\text{vir}.\dim \cX= (n-3)(1-g)+\ell+\sum_j \deg L_j  - \sum_{i,j} \Theta^i_j.
\eeq
Note that $\deg L_j=\delta_j(2g-2+\ell)/d$.
\subsection{Construction of a cosection}
Because $\gamma$ is narrow, 
we have the following useful isomorphism.

\begin{lemm}\label{hi}
Let $S$ be a scheme and let $\pi:\sC\to S$ be a flat family of twisted curve; let $\Sigma \sub\sC$ be a closed substack so that $\Sigma\to S$ is an
\'etale gerbe banded by $\mu_r$. Let $ \sL$ be a line bundle on $\sC$ so that for every $x\in \Sigma$, the homomorphism $\Aut(x)\to \Aut(\sL|_x)$ 
is via $\zeta_r\mapsto \zeta_r^k$, $1\le k<r$.  Then for each integer $1\leq c\leq k$, the homomorphism
$$\R \pi \lsta\sL(-c\Sigma)\mapright{} \R \pi\lsta \sL 
$$ 
induced by $\sL(-c\Sigma)\to\sL$ is a quasi-isomorphism.
\end{lemm}

\begin{proof}
Let $p: \sC\to C$ be the coarse moduli of $\sC$. 
 For $1\le c\le k$, we  have the exact sequence of sheaves
$$0\lra \sL(-c\Sigma)\lra \sL \lra \sL|_{c\Sigma}\lra 0.
$$
The condition $1\leq k < r$ and $1\leq c\leq k$ implies $\mathbf R \pi\lsta(\sL|_{c\Sigma})=0$, because for
$f=f(z)\in\CC[z]/(z^{c})$ the condition $f(\zeta z)=\zeta^k f(z)$ forces $f=0$.
The desired quasi-isomorphism then follows from applying $\mathbf R \pi\lsta$ to the exact sequence.
\end{proof}
 
Let $\pi_M:\cC_M\to M$, $\Sigma_{M,i}\sub\cC_M$ and $\cL_{M,j}$ be the 
universal family of $M$. We denote $\Sigma_M=\coprod_i \Sigma_{M,i}$. We suppose $\gamma$ is narrow from now on.
  
\begin{prop}\label{temp} Suppose $\gamma$ is narrow.
Then the morphism 
$$X'\defeq C(\pi_{M\ast}\sE_M(-\Sigma_{M}))\lra
\cX=C(\pi_{M\ast}\sE_M)
$$ 
induced by the inclusion $\sE_M(-\Sigma_{M})\to\cE_M$ is an
isomorphism. 
The relative perfect obstruction theory  
$$\phi_{\cX'/M}: (L\bul_{\cX'/M})\dual \lra \R \pi_{\cX\ast}\cE_\cX(-\Sigma_{\cX}) $$ 
of $X'\to M$ constructed in \cite[Prop 2.5]{CL}  coincides with (\ref{XY}) 
via the isomorphism $\R \pi_{\cX\ast}\cE_\cX(-\Sigma_{\cX})\cong \R \pi_{\cX\ast}\cE_\cX$.
 \end{prop}

\begin{proof} It follows from the construction.
\end{proof}

Because of the Proposition \ref{temp}, in the following we will not distinguish $X$ and $X'$. 
We define the relative obstruction sheaf of $X/M$ to be
\beq\label{rel-ob}
\Ob_{\cX/M}=H^1(E\bul_{\cX/M})=  R^1\pi_{\cX\ast}\cE_\cX(-\Sigma_{\cX}).
\eeq
We now construct the desired cosection
\beq\label{corel}
\sigma: \Ob_{\cX/M}\lra \sO_\cX.
\eeq
Let $S$ be a connected affine scheme. Given a morphism $S\to\cX$, we denote $\Sigma_{S,i}\sub \cC\to S$, $\cL_{S,j}$
and $\rho_S=(\rho_{S,j})\in \oplus_j\Gamma(\sL_{S,j}(-\Sigma_{S}))$ to be the pullback of the universal family on $X$, 
and letting $\Sigma_S=\sum_i \Sigma_{S,i}$.

For each monomial $W_a(x)$ of $W$, we denote $W_a(x)_j=\frac{\partial}{\partial x_j}W_a(x)$. Substituting $x_j$ by $\rho_{S,j}$, (\ref{uni}) gives
$$W_a(\rho_S)_j\defeq W_a(\rho_{S,1},\cdots, \rho_{S,n})_j\in
\Gamma(\cC, \omega\ulog_{\cC/S}\otimes \cL_{S,j}\upmo).
$$
For $\dot \rho_{S,j}\in \Gamma\bl R^1\pi_{S\ast}\cL_{S,j}(-\Sigma_{S})\br$, we define
\beq\label{sigma}\sigma(\dot \rho_{S,1}, \cdots, \dot \rho_{S,n})=
\sum_{1\le a\le m}\sum_{1\le j\le n} \alpha_a W_a(\rho_S)_j\cdot\dot\rho_{S,j}
\in H^1(\cC,   \omega_{\cC/S})\cong\Gamma(\sO_S).
\eeq
Here we used $\omega_{\cC/S}^{\text{log}}(-\Sigma_S)\cong \omega_{\cC/S}$  
and the Serre duality for orbifolds.
Because of Lemma \ref{hi}, $\Ob_{\cX/M}|_S=\oplus_{j=1}^n R^1\pi_{S\ast}\cL_{S,j}(-\Sigma_{S})$.
Thus the above construction provides us the desired cosection (homomorphism) \eqref{corel}.


Now we can define the absolute obstruction sheaf $\Ob_\cX$ of $\cX$ as follows.
Because $M$ is smooth, the projection $q: \cX\to M$ provides us a distinguinshed trinagle 
\beq\label{dt1}q^\ast\LL_M\lra \LL_\cX\lra\LL_{\cX/M}\mapright{\delta} q^\ast\LL_M[1],
\eeq
where the last term is $[q^\ast\Omega_M\to 0]$ of amplitude $[-1,0]$  .
Taking the dual of $\delta$ and composing it with the obstruction homomorphism $\phi_{\cX/M}$, we obtain the homomorphism
$q\sta\Omega_M\dual\to \Ob_{\cX/M}$. We define $\Ob_\cX$ by the exact sequence
$$0\lra q\sta\Omega_M\dual  \lra \Ob_{\cX/M}\lra \Ob_{\cX}\lra 0.
$$

\begin{prop}\label {lift}
Suppose $\gamma$ is narrow. 
Then the homomorphism $\sigma: \Ob_{\cX/M}\to\sO_\cX$ constructed 
factors through a homomorphism 
$$\bar\sigma: \Ob_\cX\lra \sO_\cX.
$$
\end{prop}


\subsection{The Proof of factorization} In this subsection, we will give a proof of the Proposition \ref{lift}. 
The proof is similar to \cite[Prop 3.5]{CL}. We first provide an equivalent construction of the cosection 
using evaluation maps.

Let $\cE_M=\oplus \cL_{M,j}$, and let
${Z}_M=\Vb(\cE_M(-\Sigma_M))$, which is the total space of the vector bundle $\cE_M(-\Sigma_M)$.
Since all $\gamma_i$'s are narrow, by Proposition \ref{temp}, we have 
$\cX= C(\pi_{M\ast}(\cE_M(-\Sigma_M)))$.
Using the universal section $\rho_\cX=(\rho_{\cX,1},\cdots\rho_{\cX,n})$ of $X$, 
and denoting $\cC_X=\cC_M\times_M X$ as the universal curve over $X$, we obtain
the evaluation (evaluating $\rho_X$) $M$-morphism 
\beq\label{ey}\ee:\cC_\cX\lra  {Z}_M.
\eeq

We form the total space of the respective vector bundles 
$\Vb(\omega_{\cC_M/M})$ and ${Z}_{i}=\Vb(\cL_{M,i}(-\Sigma_M))$.
Then the isomorphism $\Phi_a: W_a(\cL_{M,\cdot})\to \omega_{\cC_M/M}\ulog$
(cf. \eqref{uni})
and the polynomial $W=\sum \alpha_a W_a$ define an $M$-morphism
 \beq\label{morp}
h: {Z}_M={Z}_1\times_M \cdots\times_M {Z}_n\lra
\Vb(\omega_{\cC_M/M}),
\eeq
where for $\xi\in \cC_M$ and $z=(z_j)_{j=1}^n\in {Z}_M|_\xi$, 
$h(z)=\sum \alpha_a W_a(z)\in \Vb(\omega_{\cC_M/M})|_\xi$.
(Here we have used the tautological inclusion 
$\omega\ulog_{\cC_M/M}(-\Sigma_M)\to \omega_{\cC_M/M}$.)


The morphism $h$ induces a homomorphism of cotangent complexes
$$
dh: (L\bul_{{Z}_M/\cC_M})\dual
\mapright{} h^\ast 
(L\bul_{\Vb(\omega_{\cC_M/M})/\cC_M})\dual =h^\ast 
\Omega\dual_{\Vb(\omega_{\cC_M/M})/\cC_M},
$$
where $dh$ is the relative differentiation of $h$ relative to $\cC_M$.
In explicit form,  for $z=(z_j)\in {Z}_M|_\xi$ over $\xi\in \cC_{M}$,
\beq\label{dh}dh|_{z}(\dot z)=  \sum_{1\le a\le m} \sum_{1\le j\le n} \alpha_a W_a(z)_j\cdot\dot z_j,
\eeq
for $\dot z=(\dot z_j)_{j=1}^n\in 
\Omega\dual_{{Z}_M/\cC_M}\big|_z=\oplus_{j=1}^n \cL_{M,j}(-\Sigma_M)|_\xi$.

On the other hand, pulling back $d{h}$ to $\cC_{\cX}$ via 
the evaluation morphism $\ee$, we have
$$\ee^\ast (dh):\ee^\ast \Omega\dual_{{Z}_M/\cC_M}\lra 
\ee^\ast h^\ast\Omega\dual_{\Vb(\omega_{\cC_M/M})/\cC_M}.
$$
Because the right hand side is canonically isomorphic to  $\omega_{\cC_{\cX}/{\cX}}$,
applying $\R \pi_{{\cX}\ast}$, we obtain
\beq\label{map}
\begin{CD}
\sigma^\bullet:\R \pi_{{\cX}\ast}    \ee^\ast \Omega\dual_{{Z}_M/\cC_M}\lra \R \pi_{{\cX}\ast}(\ee^\ast
h^\ast \Omega\dual_{\Vb(\omega_{\cC_M/M})/\cC_M})\cong \R \pi_{{\cX}\ast}\omega_{\cC_{\cX}/{\cX}}.
\end{CD}
\eeq
  
By Proposition \ref{temp}, we obtain the canonical isomorphism 
$$E\bul_{\cX{}/M}\cong \R \pi_{\cX\ast}\cE_\cX(-\Sigma_{\cX})=\R \pi_{{\cX}\ast}    \ee^\ast \Omega\dual_{{Z}_M/\cC_M}.
$$
Hence (\ref{map}) gives 
$$\sigma^\bullet:E\bul_{\cX{}/M}\lra \R \pi_{{\cX}\ast}\omega_{\cC_{\cX}/{\cX}}.
$$

It is routine to check that $H^1(\si\bul)$ coincide with the $\sigma$ constructed in \eqref{corel}:
\beq\label{si-cosection}
\sigma=H^1(\sigma^\bullet):\Ob_{{\cX}/M}=H^1(E\bul_{{\cX}/M})\lra R^1\pi_{{\cX}\ast}(\omega_{\cC_{\cX}/{\cX}})\cong \sO_{\cX}.
\eeq

\medskip
We now prove the factorization using this interpretation of $\sigma$.
We let ${B}=C(\pi_{M\ast}\omega_{\cC_M/M})$, which by definition is
the total space of the bundle $\pi_{M\ast}\omega_{\cC_M/M}$ over $M$.  
Let $\cC_{{B}}=\cC_M\times_{M}{B}$ be the pullback of the universal curve;
let $\pi_{{B}} :\cC_{{B}}\to{B}$ be the projection. Then the universal section of ${B}$ over $\cC_{B}$ induces an evaluation morphism $\ff:\cC_{{B}}\to\Vb(\omega_{\cC_M/M})$ as (\ref{ey}).

 \begin{lemm}\label{tri}  
The following composition is trivial.
$$H^1(\sigma^{\bullet})\circ H^1( \phi_{\cX/M})=0:H^1((L\bul_{\cX{}/M})\dual )\lra H^1(E\bul_{\cX{}/M})\lra  \sO_{\cX}.
$$
\end{lemm}
 
\begin{proof} 
By Lemma \ref{hi},  the universal sections $\rho_{\cX,j}$ lie in $\Gamma(\cC_\cX,\cL_{\cX,j}(-\Sigma_{\cX}))$.  
Using (\ref{uni}),  we have
$$W_a(\rho_\cX)\defeq W_a(\rho_{\cX,1},\cdots,\rho_{\cX,n})\in \Gamma(\cC_\cX,\omega\ulog_{\cC_\cX/\cX}(-\Sigma_\cX))=
\Gamma(\cC_\cX,\omega_{\cC_\cX/\cX}).
$$
We define
$$W(\rho_\cX)=\sum \alpha_a W_a(\rho_\cX)\in \Gamma(\cC_\cX,\omega_{\cC_\cX/\cX}).
$$

The section $W(\rho_\cX)$ defines a morphism $\fg :\cX\to {B}$ so that $W(\rho_\cX)$ is the pullback of the
universal section of ${B}$ over $\cC_{B}$.
We let $\ti \fg: \cC_{\cX}\to \cC_{{B}}$ be the tautological lift of $\fg$ using
$\cC_\cX\cong \cC_M\times_M\cX$ and $\cC_{{B}}\cong \cC_M\times_M{B}$,
which fits into the following commutative square of morphisms of stacks over $\cC_M$:
\beq\label{comm}
\begin{CD}
\cC_{\cX}@>\ee>> {Z}_{M}\\
@VV{\ti \fg}V @VV{ h}V\\
\cC_{{B}}@>{\ff}>>\Vb(\omega_{\cC_M/M}),
\end{CD}
\eeq
 which in turn gives the following commutative diagrams of cotangent complexes:
\beq\label{comm2}
\begin{CD}
\pi_{\cX}\sta (L\bul_{{\cX}/{M}})\dual @= (L\bul_{\cC_{\cX}/\cC_M})\dual@>{}>> 
{\ee}^\ast\Omega\dual_{{Z}_M /\cC_M}\\
@VVV @VV{}V @VV{dh}V\\
\pi_{\cX}\sta  \fg\sta  (L\bul_{{{B}}/{M}})\dual@= \ti \fg\sta  (L\bul_{\cC_{{B}}/\cC_M})\dual@>{}>>
\ti \fg\sta \ff\sta\Omega\dual_{\Vb(\omega_{\cC_M/M})/\cC_M},
\end{CD}
\eeq
where $\pi_{\cX}\sta  \fg\sta  (L\bul_{{{B}}/{M}})\dual= \ti \fg\sta\pi_{{B}}\sta (L\bul_{{B}/M})\dual=\ti \fg\sta  (L\bul_{\cC_{{B}}/\cC_M})\dual$
follows from the fiber diagrams
\beq\label{huge}
 \begin{CD}
 \cC_\cX@>\ti \fg>> \cC_{B}@>>>\cC_M\\
 @VV{\pi_\cX}V @VV{\pi_{{B}}}V @VVV\\
 \cX@>{\fg}>> {B}@>>>M.\\ 
 \end{CD}
\eeq

Let $\phi_{{B}/M}: (L\bul_{{B}/{M}})\dual\to \Rpi_{{B}\ast} 
\omega_{\cC_{{B}}/{B}}$
be the standard obstruction theory of ${B}\to M$ (cf. \cite{BF, CL}).
Then $\fg\sta\phi_{{B}/M}$ is the obstruction theory of ${B}\times_M\cX\to \cX$. The smoothness of ${B}\to M$ implies
\beq\label{3.11}
0=\ H^1(\fg\sta\phi_{{B}/M}): H^1(\fg\sta  (L\bul_{{B}/{M}})\dual )
\lra \fg\sta R^1\pi_{{B}\ast} \omega_{\cC_{{B}}/{B}}.
\eeq
 
Finally, applying $R^1\pi_{\cX\ast}$ to  \eqref{comm2},  we see that
the composite
\begin{eqnarray}\label{exact1}
&&H^1((L\bul_{\cX/M})\dual)\mapright{} R^1\pi_{\cX\ast}\ee^\ast\Omega\dual_{{Z}_{M}/\cC_M}
\lra R^1\pi_{\cX\ast}\ee^\ast  h^\ast \Omega\dual_{\Vb(\omega_{\cC_M/M})/\cC_M}
\end{eqnarray}
coincides with the composite
\begin{eqnarray}\label{exact2}
&&H^1((L\bul_{\cX/M})\dual )\lra H^1(\fg\sta (L\bul_{{B}/M})\dual )\mapright{0}
\fg\sta R^1\pi_{{B}\ast} \ff\sta \Omega\dual_{\Vb(\omega_{\cC_M/M})/\cC_M}.
\end{eqnarray} 
Using \eqref{3.11}, the composite in (\ref{exact2}) is trivial, thus the composite in (\ref{exact1}) is trivial.
Because
$\ee^\ast  h^\ast \Omega\dual_{\Vb(\omega_{\cC_M/M})/\cC_M}= \omega_{\cC_{\cX}/\cX}$,
we prove the desired vanishing.
\end{proof}

\begin{proof}[Proof of Proposition \ref{lift}]
By (\ref{si-cosection}), $\sigma=H^1(\sigma^\bullet)$. Hence the composition of $\sigma $ 
with $q\sta\Omega_M\dual\to \Ob_{\cX/M}$ (defined below (\ref{dt1})) is the $H^1$ of the composition 
$$q\sta(\LL_{M})\dual[-1]\lra (\LL_{\cX/M})\dual\mapright{\phi_{\cX/M}}E\bul_{\cX/M}\mapright{\sigma^\bullet} \Rpi_{\cX\ast}\omega_{\cC_\cX/\cX},
$$
where the first arrow is the $\delta\dual$ in \eqref{dt1}. Lemma \ref{tri} implies the $H^1$ of the above composition is trivial.
This proves the Proposition.
\end{proof}

\subsection{Degeneracy loci of the cosection}

We investigate the loci of non-surjectivity of the cosection $\sigma$.
Let $\xi=(\Sigma^\sC_i\sub \sC,\sL_j, \rho_j)$ be a closed point in $\cX=\barM_{g,\gamma}(G)^p$. 

\begin{lemm}\label{Dege}
Let the notation be as stated. Then $\sigma|_\xi=0$ if and only if all $\rho_j=0$. Thus the (reduced part of) degeneracy loci of $\sigma$ is $M\sub\cX$, which is proper.
\end{lemm}

\begin{proof} $\sigma|_\xi=0$ implies that for arbitrary $\dot \rho_{1}, \cdots, \dot \rho_{n}\in H^1(C, L_j)$,
the term in (\ref{sigma}) 
\begin{eqnarray*}\sigma(\dot \rho_{1}, \cdots, \dot \rho_{n})=
\sum_{1\le a\le m}\sum_{1\le j\le n} \alpha_a W_a(\rho)_j\cdot\dot\rho_{j}
\in\Gamma(C,\omega_{C})\cong\CC
\end{eqnarray*}
vanishes. By Serre duality 
this forces $\sum_{1\le a\le m} \alpha_a W_a(\rho)_j=0$
for every $j$. Since $0\in \CC^n$ is the only critical point of $W$, this forces $(\rho_1,\cdots,\rho_n)=0$.
\end{proof}

\subsection{Localized virtual cycle}
We recall the notion of kernel-stack of a cosection. Let $E=[E^0\to E^1]$ be a
two term complex of locally free sheaves on a Deligne-Mumford stack $X$; let
$f: H^1(E)\to \sO_X$ be a cosection of $H^1(E)$. We define $D(f)$ to be the
subset of $x\in X$ such that $f|_x=0: H^1(E)|_x\to\CC$; $D(f)$ is closed. Let $U=X-D(f)$. 

\begin{defi}\label{cone-stack-def}
Let the notation be as stated. We define the kernel stack to be 
$$h^1/h^0(E)_f\defeq \bl h^1/h^0(E)\times_X D(f)\br \cup \ker\{h^1/h^0(E)|_U\to H^1(E)|_U\to \CC_U\}.
$$
\end{defi}

Here $h^1/h^0(E)|_U\to H^1(E)|_U$ is the tautological projection and $H^1(E)|_U\to \CC_U$ is
$f|_U$. Since $f$ is surjective over $U$, the composite in the bracket is surjective, thus the kernel is
a bundle-stack over $U$. Clearly, the union is closed in $h^1/h^0(E)$; we endow it the reduced structure,
making it, denote by $h^1/h^0(E)_f$, a closed substack of $h^1/h^0(E)$.
We call it the kernel-stack of $f$
 
We apply the theory developed in \cite{KL} to $\cX/M$ for $\cX=\barM_{g,\gamma}(G)^p$ and $M= \barM_{g,\gamma}(G)$.
As $\sigma$ is a cosection of $H^1(E^\bullet_{\cX/M})$, 
we form its kernel-stack
\beq\label{cone-stack}h^1/h^0(E^\bullet_{\cX/M})_{\sigma}\sub h^1/h^0(E^\bullet_{\cX/M}).
\eeq

\begin{prop}\label{loc-vir-sig1} The virtual normal cone cycle 
$$[\bC_{\cX/M}]\in Z\lsta (h^1/h^0(E^\bullet_{\cX/M}))
$$
lies inside $Z\lsta (h^1/h^0(E^\bullet_{\cX/M})_{\sigma})$. 
\end{prop}

\begin{proof}
The smoothness of the morphism from $\bC_{\cX/M}$ to $\bC_{\cX}$(the intrinsic normal cone of $\cX$) and Propostion \ref{lift} 
reduce the Proposition to the absolute case, which is proved in \cite[Prop. 4.3]{KL}.
\end{proof}

Following \cite{KL}, we form the localized Gysin map
$$0^!_{\sigma,\mathrm{loc}}: A\lsta( h^1/h^0(E^\bullet_{\cX/M})_{\sigma})\lra A_{\ast-\mathrm v}( D(\sigma)),
$$
where $\mathrm v=\text{vir}.\dim X-\dim M$ and
$\text{vir}.\dim X$ is given in (\ref{vd}). Recall $D(\sigma)=M\sub \cX$ (cf. Lemma \ref{Dege}).

\begin{defi-prop}\label{defi} Let $([\CC^n/G],W)$ be an LG space. For $g\ge 0$ and 
a $g$-admissible $\gamma$, we define the Witten's top Chern class of $\barM_{g,\gamma}(G)$ to be 
$$[\barM_{g,\gamma}(G)^p]\virt_{\sigma}:=
0^!_{\sigma,\mathrm{loc}}([\bC_{\barM_{g,\gamma}(G)^p/\barM_{g,\gamma}(G)}])
\in A_{\ast} (\barM_{g,\gamma}(G)),
$$
for the cosection $\sigma$ constructed using $W$.
It is independent of the choices of $W$.
 \end{defi-prop}
 
\begin{proof}  We only need to prove the independence on  $W$.
Suppose $([\CC^n/G],W_0)$ and $([\CC^n/G],W_1)$ are two LG spaces of the same weight $(d,\delta)$.
Let $W_t=tW_1+(1-t)W_0$, $t\in \Ao$. Then there is a Zariski open $U\sub \Ao$ 
containing $0,1$ such that $W_t$ is nondegenerate for $t\in U$. 
Then every $W_t$, $t\in U$, induces a cosection $\sigma_t$ of $\Ob_{\barM_{g,\gamma}(G)^p}$.
Indeed, if $\sigma_0$ and $\sigma_1$ are the cosections constructed using $W_0$ and $W_1$, then
$\sigma_t=t\sigma_1+(1-t)\sigma_0$. 

For $t\in U$, since $W_t$ is non-degenerate, the degeneracy loci of $\sigma_t$ is $M=\barM_{g,\gamma}(G)$. 
The family $\sigma_t=t\sigma_1+(1-t)\sigma_0$ forms a family of cosections of the family moduli space
$U\times \barM_{g,\gamma}(G)^p$. As this family is a constant family, \cite[Thm 5.2]{KL} applies and hence 
$[\barM_{g,\gamma}(G)^p]\virt_{\sigma_t}\in A\lsta (\barM_{g,\gamma}(G))$ is independent of $t$.
 \end{proof}
 
Because of the independence to the choice of $W$, the class $[\barM_{g,\gamma}(G)^p]\virt_\sigma$ only depends on 
$G\le \Gm^n$ and the weight $(d,\delta)$. In the following
we will drop the subscript $\sigma$, and denote the Witten's
top Chern class of $([\CC^n/G],W)$ to be
$$[\barM_{g,\gamma}(G)^p]\virt\in A\lsta (\barM_{g,\gamma}(G)).
$$ 

\section{Witten's top Chern class of strata}
\def\FF{F\bul}
\def\EE{E\bul}
\def\cXX{\mathcal X}

In this section, we fix an LG space $([\CC^n/G],W)$. We also fix $g$, $\ell$, and
a $g$-admissible $\gamma=(\gamma_i)_{i=1}^\ell$. Let $M=\barM_{g,\gamma}(G)$ be the moduli of $G$-spin curves,
and let $\cX=\barM_{g,\gamma}(G)^p$ 
be the moduli of $G$-spin curves with fields. 
As before, we denote by $(\pi_M: \cC_M\to M, \cL_{M,j})$ (part of) the universal family of $M$, and denote 
$\cE_M=\oplus_{j=1}^n \cL_{M,j}$. 

\subsection{Virtual cycles and Gysin maps}
We recall a general fact about the cosection localized virtual cycles and Gysin maps.

Let $\cM$ be a smooth DM stack; let $\cXX$ be a DM stack and let $\cXX\to \cM$ be a representable morphism.
We assume $\cXX\to \cM$ has a relative perfect obstruction theory $\phi_{\cXX/\cM}: (\LL_{\cXX/ \cM})\dual\to F\bul$;
we define its relative obstruction sheaf to be $\Ob_{\cXX/\cM}:=H^1(F^{\bullet})$. 
Let $\sigma:\Ob_{\cXX/\cM}\to \sO_\cXX$ be a cosection so that its 
composition with $\Omega\dual_{\cM}\to \Ob_{\cXX/\cM}$ is trivial. Then by
 \cite[Thm 5.1]{KL}, denoting $\bC$ ($=C_{\cXX/\cM}$) the normal cone of $\cXX/\cM$,
and $0^!_{\sigma,loc}$ the localized Gysin map defined in \cite{KL},
the $\sigma$-localized virtual cycle of $\cXX$ is 
$$[\cXX]\virt:=0^!_{\sigma,loc}[\bC]\in A\lsta (D(\sigma)).
$$

Let $\iota:\cS\to \cM$ be a proper representable l.c.i. (or flat) morphism between DM stacks of constant codimension. 
We form the Cartesian square
$$
\begin{CD}
\cY@>g>>\cXX\\
@VVV@VVV\\
\cS@>\iota>>\cM\,.
\end{CD}
$$
Since $\cXX\to \cM$ is representable, $\cY$ is a DM stack.
The obstruction theory of $\cXX\to\cM$ induces a perfect relative obstruction theory of $\cY\to\cS$ by pullback \cite[Prop 7.2]{BF}\footnote{The
assumption that $\iota:\cS\to\cM$ is l.c.i. is sufficient for \cite[Prop 7.2]{BF} to be valid.}:
$$\phi_{\cY/\cS}:( \LL_{\cY/\cS})\dual\lra \EE:=g\sta \FF ;
$$
the cosection $\sigma$ pulls back to a cosection $\sigma':\Ob_{\cY/\cS}\to \sO_\cY$
whose degeneracy loci $D(\sigma')=D(\sigma)\times_\cM \cS$.
Since $\iota$ is proper, $D(\sigma')$ is proper if $D(\sigma)$ is proper.
Further, the composite of $\Omega\dual_\cS\to \Ob_{\cY/\cS}$
with $\sigma':\Ob_{\cY/\cS}\to\sO_\cM$ vanishes because of the vanishing assumption
on the similar composition on $\cXX$. Appying \cite[Thm 5.1]{KL}, using the virtual normal cone $\bC'=C_{\cY/\cS}$
of $\cY\to\cS$, we obtain the
$\sigma'$-localized virtual cycle 
\beq\label{M}[\cY]\virt:=0^!_{\sigma',loc}[\bC']\in A\lsta (D(\sigma')).
\eeq

\begin{lemm}\label{sub-cycle} We have identity $\iota^![\cXX]\virt=[\cY]\virt$ 
\end{lemm}
 
\begin{proof} If $\iota$ is flat, $\iota\sta\bC_{\cX/\cM}=\bC_{\cY/\cM}$ and the identity follows from functorial property of cosection localization.
We now consider the case where $\iota$ is an l.c.i. morphism. We denote the normal sheaf
$N=N_{\cS/\cM}$; it is a bundle stack over $\cS$ because $\iota$ is an l.c.i morphism. 
A rational equivalence $R\in W\lsta(\bC\times_\cM N)$ was  constructed in \cite{Kr,Vi} such that
$\partial R=[C_{g\sta \bC/\bC}]-[\bC'\times_\cS N]$.
Let 
$$\ti\sigma:h^1/h^0(g\sta F^{\bullet})\times_\cM N\to \sO_\cM
$$
be the lift of the pair (of the induced) $\sigma_\cS:h^1/h^0(g\sta F^{\bullet})\to \sO_\cM$ and $0: N\to \sO_\cM$.
Then the degeneracy loci of $\ti\sigma$ is identical to the degeneracy loci of $\sigma'$. 
The standard property of localized Gysin maps states that
\beq\label{pullback} \iota^![\cXX]\virt=\iota^!0^!_{\sigma,loc}[\bC]=0^!_{\ti\sigma,loc}[C_{g\sta \bC/\bC}].\eeq 
Since $\sigma$ annihilates the reduced part of $\bC$, the reduced part of the stack $\bC \times_\cM N $ is also annihilated by $\ti\sigma$ (i.e. lies in the
kernel stack of $\ti\sigma$).
Thus $R$ lies in the kernel stack of $\ti\sigma$, and we have
$0^!_{\ti\sigma,loc}[C_{g\sta \bC/\bC}]=0^!_{\ti\sigma,loc}[\bC'\times_\cS N]=0^!_{\sigma',loc}[\bC']$,
which is $[\cY]\virt$. This proves the Lemma.
\end{proof}

 

We will apply  Lemma \ref{sub-cycle} to the pair $X=\barM_{g,\gamma}(G)^p\to M=\barM_{g,\gamma}(G)$,
and some l.c.i. morphism $S\to M$ whose image is the closed substack defined by the topological (or geometrical) type of the $G$-spin curves in $M$.

 \subsection{Topological strata of the moduli of $G$-spin curves} 
In this subsection, we will associate to a pointed $G$-spin curve its decorated dual graph.  For such a dual graph, we will construct its
associated stratum in the moduli space. 
We will rephrase the constructions in \cite[Sect 2.2.2 and 2.2.3]{FJR2}.

Given a pointed $G$-spin curve $(\sC,\Sigma^\sC_i,\sL_j,\varphi_k)$ (not necessarily connected), following the standard procedure
we associate to it a decorated (dual)
graph $\Gamma$ as follows\footnote{Our convention is that every edge is directed and has two vertices;
every tail has only one vertex.}.
 Its vertices $v\in V(\Gamma)$ correspond to irreducible components $\sC_v\sub \sC$;
its edges $e\in E(\Gamma)$ correspond to nodes $q_e\in\sC$ so that a vertex $v$ of $e$
has $q_e\in\sC_{v}$; its tails $t_i\in T(\Gamma)$ correspond to the 
markings $\Sigma_i^\sC$ so that the vertex $v$ of $t_i$ has $\Sigma_i^\sC\sub \sC_v$;
we give $e\in E(\gamma)$ a direction that is consistent with a labeling of its two vertices $v_{e}^-$ and $v_{e}^+$ 
(i.e. from $-$ to $+$).


We add decorations to $\Gamma$. 
For a (directed) edge $e$, in case $v_e^-\ne v_e^+$,
there is a unique labeling of the two branches of (the formal completion of $\sC$ along $q$) $\hat \sC_q$ as 
$\hat\sC_{e-}$ and $\hat\sC_{e+}$
so that $\hat\sC_{e+}$ is the formal completion of $\sC_{v_e^+}$ along $q$;
in case $v_e^-= v_e^+$, we fix a labeling of the two branches of $\hat \sC_q$ into $\hat\sC_{e-}$ and $\hat\sC_{e+}$.
Under this agreement, we decorate $e\in E(\Gamma)$ by the monodromy representation 
$\gamma_e\defeq \gamma_{q_e+} : \mu_{r_{q_e}}\to G$.
We decorate the tail $t_i$
by $\gamma_i$, the monodromy representation along $\Sigma_i^\sC$.
We decorate every vertex $e$ by $g_e=g(\sC_e)$, the arithmetic genus of  $\sC_e$.

We also define the coarse dual graph of $(\sC,\Sigma^\sC_i,\sL_i,\varphi_j)$ to be a no-edge
graph so that its vertices $e\in E(\Gamma)$ correspond to connected components $\sC_e\sub \sC$, decorated by 
$g_e=g(\sC_v)$; its tails and decorations are as before.
Given a decorated graph $\Gamma$, by reversing the direction of one edge $e\in E(\Gamma)$
and replace its decoration $\gamma_e$ by $\gamma_e\upmo$, we obtain a new decorated graph. We call this
process ``direction reversing" operation. We say two decorated graphs are equivalent if one is derived from the other
by repeated ``direction reversing" operations.

In general, a $G$-decorated graph 
is one that is the dual graph of a not necessary connected $G$-spin curve. 


\begin{defi}
We say a pointed $G$-spin curve $(\sC,\Sigma^\sC_i,\sL_j,\varphi_k)$ is labeled by a $G$-decorated
graph $\Gamma$ if there are
bijections between the sets of irreducible components (resp. nodes; resp. markings) 
of $(\sC,\Sigma^\sC_i)$ with $V(\gamma)$ (resp. $E(\Gamma)$; resp. $T(\Gamma)$) that make $\Gamma$ a decorated dual
graph of $(\sC,\Sigma^\sC_i,\sL_j,\varphi_k)$. 

In case $\Gamma$ has no edges, we say the family is 
coarsely labeled by the $\Gamma$ if we replace ``irreducible components" by ``connected components".
\end{defi}

Given a $G$-decorated graph $\Gamma$, we define its complete splitting to be
the graph resulting from breaking every (directed) edge $e\in E(\Gamma)$ into a pair of tails, denoted by
$t_{e-}$ and $t_{e+}$, attached to the vertices $v_e^-$ and $v_e^+$, and
assigned the labeling $\gamma_e\upmo$ and $\gamma_e$, respectively.
We denote the complete splitting of $\Gamma$ by $\Gamma_{sp}$. 

%
%

\begin{defi}\label{4.3}
Let $\Gamma$ be a $G$-decorated graph and $\Gamma_{sp}$ its complete splitting. For a
scheme $S$, we say an $S$-family $(\cC,\Sigma_i^{\cC},\sL_{j}, \varphi_{k})$ 
of pointed $G$-spin curves is weakly labeled by ${\Gamma}$ if the following hold:
\begin{enumerate}
\item there is a family of $G$-spin curves $(\cC',\Sigma_i^{\cC'},\sL'_{j}, \varphi'_{k})$ 
coarsely labeled by $\Gamma_{sp}$;
\item there is a morphism $\cC'\to\cC$ so that $\cC$ is the result of identifying
all pairs of marked sections of $(\cC',\Sigma_i^{\cC'})$ associated with $t_{e-}$ and $t_{e+}$, for all $e\in \Gamma$;
\item under $\cC'\to\cC$, $(\sL'_{j}, \varphi'_{k})$ is the pullback of $(\sL_{j}, \varphi_{k})$,
and $\Sigma_i^{\cC}$ is identified with $\Sigma_i^{\cC'}$ for each $t_i$ , the $i$-tail of $\Gamma$.
\end{enumerate}
\end{defi}


We recall the associated operation on the dual graphs by smoothing nodes of $G$-spin curves. Let $\Gamma$ be a $G$-decorated
dual graph and $e\in E(\Gamma)$. We define $\Gamma/e$ to be the graph resulting by eliminating the edge $e$ from
$\Gamma$; merging the vertices $v_e^-$ with $v_e^+$ to a single vertex, denoted by $\overline{v}$, and decorate it by 
$g_{\bar v}=g_{v_e^-}+g_{v_e^+}$ when $v_e^-\ne v_e^+$, and $g_{\bar v}=g_{v_e^-}+1$ when $v_e^-=v_e^+$.
We call $\Gamma/e$ the contraction of $\Gamma$ by $e$. This process also 
defines a graph map $\Gamma\to \Gamma/e$ that is the identity except that it sends $e$, $v_e^\pm$ to $\bar v\in V(\Gamma/e)$.

We call $\Gamma'$ an edge-contraction of $\Gamma$ if $\Gamma'$ results from $\Gamma$ by repeatedly contracting edges.
Let $\Gamma\to\Gamma'$ be the compositions of the contraction maps; we define $\Aut(\Gamma/\Gamma')$ to be the
automorphisms of $\Gamma$ that commute with the identity of $\Gamma'$ via the map $\Gamma'\to \Gamma$,
up the equivalences (defined by edge-direction reversing operation).

We form the category of families of weakly $\Gamma$-labeled $G$-spin curves; it is a DM stack, denoted by
$\barM_\Gamma(G)$. In case $\Gamma'$ is an edge-contraction of $\Gamma$,
we have a tautological morphism
$$\tau_{\Gamma\Gamma'}: \barM_{\Gamma}(G)\lra \barM_{\Gamma'}(G).
$$

In case $\Gamma$ is connected and $\Gamma'$ consists of a single vertex, decorated by $g$, then this morphism
takes the form
$\tau_\Gamma: \barM_{\Gamma}(G)\lra \barM_{g,\gamma}(G)$,
where $\gamma=(\gamma_i)$ is the collection of decorations of tails of $\Gamma$, consistent with the ordering of tails.

%
%
%


\begin{lemm}\label{emb} 
Suppose $\Gamma'$ is an edge-contraction of $\Gamma$.
Then the morphism $\tau_{\Gamma\Gamma'}$ is a finite l.c.i. morphism. We define
$\barM_{\Gamma'}(G)_{\Gamma}$ to be the image stack of $\tau_{\Gamma\Gamma'}$. Then
the factored
$\barM_{\Gamma}(G)\to \barM_{\Gamma'}(G)_{\Gamma}$  has pure degree $|\Aut(\Gamma/\Gamma')|$. 
\end{lemm} 

\begin{proof} 
By induction, we only need to prove the case when $\Gamma'$ is obtained by contracting one edge of $\Gamma$.
Let $\mathfrak M_{g,\ell}^{\rm tw,\circ}$ be the moduli stack of genus $g$ stable curves with $\ell$-twisted
markings and no twisting at nodal points. 
The structure of the  forgetful morphism (eliminating the stacky structure along nodes) $\mathfrak M_{g,\ell}^{\rm tw}\to \mathfrak M_{g,\ell}^{\rm tw,\circ}$ was studied in details by Olsson \cite{Ol} (see \cite{Ch1} \S 2.4 as well). Similar to the situation studied by Chiodo and Ruan in \cite{CR},
$\barM_{g, \gamma}(G)$ is \'etale over $\mathfrak M_{g, \ell}^{\rm tw}$. Therefore we can and will use Olsson's result in \cite{Ol} in the following.

Let $\xi= (\sC,\Sigma^\sC_i,\sL_j,\varphi_k)\in \barM_{g,\gamma}(G)(\CC)$. Suppose $\sC$ has nodal
points $p_1,\ldots, p_m$ with stabilizer groups $\mu_{r_i}$ at $p_i$. 
By \cite[Remark 1.10]{Ol} (see \cite{Ch1} \S 2.4 and Remark 2.4.10 as well), we can find an affine \'etale neighbourhood $S\to 
\mathfrak M_{g,\ell}^{\rm tw,\circ}$ so that an open neighborhood of the product
$$\bar\xi\in U\sub \barM_{g, \gamma}(G)\times_{\mathfrak M_{g,\ell}^{\rm tw,\circ}}S,
$$
where $\bar \xi$ lies over $\xi$, is 
$$U=\big[\spec A/\mu_{r_1}\times\ldots\times\mu_{r_m}\big], \quad A=\sO_S[x_1, \ldots, x_m]/(x_i^{r_i}=u_i,i=1,\cdots, m),
$$
where $u_i\in \sO_S$ defines the divisor of curves in $S$ of which the node $p_i$ in the coarse moduli of $\sC$
remains nodal; $\mu_{r_i}$ acts on $A$ via $x_j^{\zeta_{r_i} }=x_j$ when $j\ne i$
and $=\zeta_{r_i}x_i$ when $j=i$. By \cite[Remark 1.10]{Ol}, $\{x_i=0\}$ is the divisor where the node $p_i$
remains nodal. 

Assume $\xi\in \barM_{\Gamma}(G)$. Since $\Gamma\to\Gamma'$ is by contracting one edge, 
by re-indexing, we can assume that it contracts the edge associated to the node $p_m$.
By shrinking $S$ if necessary, we have
isomorphisms
$$
\bar\xi\in U\cap \bl \barM_{\Gamma'}(G)\times_{\mathfrak M_{g,\ell}^{\rm tw,\circ}}S\br
=\big[\spec\bl A/(x_1, \ldots, x_{m-1})\br/\mu_{r_1}\times\ldots\times\mu_{r_{m-1}}\big],
$$
and
$$
\bar\xi\in U\cap \bl \barM_{\Gamma}(G)\times_{\mathfrak M_{g,\ell}^{\rm tw,\circ}}S\br
=\big[\spec\bl A/(x_1, \ldots, x_{m})\br/\mu_{r_1}\times\ldots\times\mu_{r_m}\big].
$$
Therefore the morphism $\tau_{\Gamma\Gamma'}$, restricted to 
$\barM_{\Gamma}(G)\times_{\mathfrak M_{g,\ell}^{\rm tw,\circ}}S$, is \'etale covered by the morphism
$$\spec\bl A/(x_1, \ldots, x_{m})\br\lra
\spec\bl A/(x_1, \ldots, x_{m-1})\br,
$$
which is an immersion of a Cartier divisor.
This proves that $\tau_{\Gamma\Gamma'}$ is a finite l.c.i. morphism.

Finally, the degree of $\tau_{\Gamma\Gamma'}$ is $|\Aut(\Gamma/\Gamma')|$ can be verified over generic point of 
$ \barM_{\Gamma'}(G)_{\Gamma}$; we omit the proof. 
\end{proof}
    
\subsection{The forgetful morphism}

In GW-theory, by removing the last marked point and stabilizing, we obtain the forgetful morphism
$\barM_{g,\ell}\to\barM_{g,\ell-1}$. Similar morphism exists for $\barM_{g,\gamma}(G)$ if $\gamma_\ell$
in $\gamma$ takes the form (cf. Subsection 2.1)
\beq\label{gd} \gamma_\ell: \mu_d\lra G,\quad \zeta_d\mapsto \jmath_\delta.
\eeq
In general, let $\Gamma$ be a $G$-decorated graph whose $\ell$-tails are decorated by $\gamma=(\gamma_i)_{i=1}^\ell$.
Let $\Gamma'$ be the stabilization of the graph after removing (the tail) $t_\ell$ from $\Gamma$; thus the tails of 
$\Gamma'$ are decorated by $\gamma'=(\gamma_i)_{i=1}^{\ell-1}$.

\begin{theo}[Forgetting tails]\label{fgtail}
Let $\Gamma$ and $\Gamma'$ be as stated. Suppose $\gamma_\ell$
takes the form \eqref{gd}.
Then there is a flat forgetful morphism 
$$\ff_{\Gamma,\ell}:\barM_\Gamma(G)\lra \barM_{\Gamma'}(G)
$$
that sends $(\sC,\Sigma_i^\sC,\sL_j,\varphi_k)$ to $(\sC',\Sigma_{j<\ell}^\sC, \sL'_{j}, \varphi_k')$, where 
$\cC'$ results from $\cC$ by removing the marked section $\Sigma^{\cC}_\ell$, eliminating the stacky structure
along $\Sigma^\cC_\ell$, and stabilizing $(\sC', \Sigma_{j<\ell}^{\sC})$; $\sL_j'$ is the pushforward of $\sL_j$
under the tautological map $\cC\to\cC'$. 
\end{theo}

Let $\ti\gamma=(\gamma_1,\cdots,\gamma_{\ell-1},1)$, where $1$ is the trivial representation $\{1\}\to G$.
Before proving the Theorem, we construct an auxiliary stack $\barM_{g,\ti\gamma[\ell]}(G)$ that is the groupoid of all families
$(\sC,\Sigma_i^{\sC}, \sL_j,\varphi_k)$ obeying all requirements in Definition \ref{W-curves} and \ref{def-2.7}
except that $\gamma$ is replaced by $\ti\gamma$ and (\ref{Ow}) is replaced by
\beq\label{phip}
\phi_k\colon
\bm_k(\sL_1, \ldots, \sL_n)\mapright{\cong}(\omega_{\sC}^{\text{log}}(-\Sigma_\ell^{\cC}))^{w(\bm_k)}  ,\quad k=1,\cdots,n.  
\eeq
(Note that $\ti\gamma_\ell$ is trivial implies that $\Sigma_\ell^{\cC}\sub\cC$ is non-stacky.)

\begin{lemm}
Let the notation be as stated. Then $\barM_{g,\ti\gamma[\ell]}(G)$ is a smooth proper DM stack over $\CC$.
\end{lemm}

\begin{proof}
Following \cite[Thm 2.2.6]{FJR2} and \cite[Sect 1.5]{AJ}, we will transform its construction to the existence of moduli
of twisted stable morphisms. Recall that given a twisted curve $\sC$, a line bundle $\sL$ on $\sC$ induces a morphism
$[\sL]: \sC\to B\Gm$. Thus given an $S$-family $\xi=(\sC,\Sigma_i^\sC,\sL_j,\phi_k)\in \barM_{g,\ti\gamma[\ell]}(G)(S)$, the
$n$-tuple $(\sL_j)_{j=1}^n$ defines a morphism $\tau(\xi)$ as shown in the following diagram.
Let $(C,\Sigma_i^C)$ be the coarse moduli of $(\sC,\Sigma_i^\sC)$. Abbreviating $\omega_{S,\ell}=\omega_{C/S}^{log}(-\Sigma_{S,\ell}^C)$, and $w_k=w(\bm_k)$, the $n$-tuple of line bundles $(\omega_{S,\ell}^{w_k})_{k=1}^n$ induces a morphism $B(\omega_{S,\ell}^{w_\cdot})$ as shown
below. Further, let $\bm: \Gm^n\to\Gm^n$ be the homomorphism defined via $t=(t_1,\cdots, t_n)\mapsto (\bm_1(t),\cdots, \bm_n(t))$,
and $B\bm: B\Gm^n\to B\Gm^n$ be its induced morphism. Then the isomorphisms $\phi_1,\cdots,\phi_n$  in \eqref{phip} fit into the
(left hand side) commutative square:
$$
\begin{CD}
\sC @>>> C @>{\rho}>> C_{g,\ell-1}\\
@VV{\tau(\xi)}V @VV{B(\omega_{S,\ell}^{w_\cdot})}V @VV{B(\omega_{g,\ell-1}^{w_\cdot})}V\\
B\Gm^n @>{B\bm}>> B\Gm^n @= B\Gm^n
\end{CD}
$$
The right hand side square is constructed as follows. Let $(C,\Sigma_{i<\ell}^C)\to (C,\Sigma_{i<\ell}^C)^{st}$ be the
stabilization of the family of $(\ell-1)$-pointed curves $(C,\Sigma_{i<\ell}^C)$; let $(C,\Sigma_{i<\ell}^C)^{st}\to (C_{g,\ell-1}, \Sigma_i^{C_{g,\ell-1}})$
be the tautological morphism to the universal family 
of the moduli of $(\ell-1)$-pointed genus $g$ stable curves. 
We let $\rho: C\to C_{g,\ell-1}$ be the composite of the stabilization and the tautological morphisms mentioned.
Let $\omega_{g,\ell-1}=\omega_{C_{g,\ell-1}/\barM_{g,\ell-1}}^{log}$.
Because $\rho\sta\omega_{g,\ell-1}=\omega_{S,\ell}$, 
the morphism $B(\omega_{g,\ell-1}^{w_\cdot})$ induced by the $n$-tuple of line bundles $(\omega_{g,\ell-1}^{w_k})^n_{k=1}$
makes the right hand side square commutative.

In conclusion, the two squares define a morphism
\beq\label{Fxi}
F(\xi): 
\sC\lra \sC_{g,\ell-1,\bm}\defeq C_{g,\ell-1}\times_{B\Gm^n}B\Gm^n
\eeq
such that\\
(1) $F(\xi)$ is an $S$-family of balanced $\ell$-pointed genus $g$  twisted stable morphism with the fundamental class to be the fiber class of $C_{g,\ell-1}/\barM_{g,\ell-1}$; \\
(2) $F(\xi)(\Sigma_i^{\sC})$ lies in the sector  $\Sigma_i^{\sC_{g,\ell-1}}\defeq \Sigma_i^{C_{g,\ell-1}}\times_{B\Gm^n} B\Gm^n$
given by the representation $\gamma_i$ for $1\le i\le \ell-1$;\\
(3) the last marked section $\Sigma_\ell^\sC\sub \sC$ is non-stacky.
  
We let $\sK^{bal}_{g,\ti\gamma}(\sC_{g,\ell-1,\bm},F)$ be the groupoid  of balanced $\ell$-pointed genus $g$ twisted stable morphisms
to $\sC_{g,\ell-1,\bm}$ of fundamental classes $F$ satisfying conditions (2) and (3) above. 
Following the 
study in \cite{FJR2} and \cite{AJ}, the groupoid $\sK^{bal}_{g,\ti\gamma}(\sC_{g,\ell-1,\bm},F)$ forms 
a smooth proper DM stack  over $\CC$.
The above correspondence constructs a transformation of groupoid
\beq\label{eq5}
\barM_{g,\ti\gamma[\ell]}(G)\lra \sK^{bal}_{g,\ti\gamma}(\sC_{g,\ell-1,\bm},F).
\eeq

Conversely, given an $S$-family in $\sK^{bal}_{g,\ti\gamma}(\sC_{g,\ell-1,\bm},F)$, following \cite{FJR2}, 
it produces an $S$-family $\xi$ in $\barM_{g,\ti\gamma[\ell]}(G)$ whose induced $F(\xi)$ is the given $S$-family in
$\sK^{bal}_{g,\ti\gamma}(\sC_{g,\ell-1,\bm},F)$. This shows that this transformation is an equivalence. This proves that
$\barM_{g,\ti\gamma[\ell]}(G)$ is a smooth proper DM stack  over $\CC$.
\end{proof}

\begin{proof}[Proof of Theorem \ref{fgtail}]
Let $\sK^{bal}_{g,\gamma'}(\sC_{g,\ell-1,\bm},F)$ be the moduli of $(\ell-1)$-pointed genus $g$ balanced twisted stable curves
to $\sC_{g,\ell-1,\bm}$ of fundamental classes $F$ that satisfy the requirement (2) above. It is a smooth proper DM stack  over $\CC$
(\cite[Thm 2.2.6]{FJR2}). Let
$$\ff: \sK^{bal}_{g,\ti\gamma}(\sC_{g,\ell-1,\bm},F)\lra \sK^{bal}_{g,\gamma'}(\sC_{g,\ell-1,\bm},F)
$$
be the forgetful morphism (forgetting the $\ell$-th marked points followed by stabilization) in \cite[Coro 9.1.3]{A-V}. 
By the construction in \cite{FJR2}, canonically $\sK^{bal}_{g,\gamma'}(\sC_{g,\ell-1,\bm},F)\cong \barM_{g,\gamma'}(G)$. 
Combined with the isomorphism \eqref{eq5}, we obtain the morphism 
$$\fv_1: \barM_{g,\ti\gamma[\ell]}(G)\to \barM_{g,\gamma'}(G).
$$

Next, we construct a morphism $\fv_2:\barM_{g,\gamma}(G)\to \barM_{g,\ti\gamma[\ell]}(G)$.  
Given any family $\xi=(\sC, \Si_i^\sC, \sL_j,\varphi_k)$ in $\barM_{g,\gamma}(G)$, we construct a new family
$\xi'=(\sC', \Si_i^{\sC'}, \sL_j',\varphi_k')$ as follows. We let $\sC'$ be $\sC$ with the (possible) stacky structure 
along $\Si_\ell^\cC$ eliminated. Let $\ft: \sC\to\sC'$ be the tautological morphism; we define
$\sL_j'=\ft\lsta \sL_j$. Because $\gamma_\ell$ is of the form $\zeta_d\mapsto \jmath_\delta$, 
one verifies that the isomorphism $\varphi_k$ induces isomorphism (\ref{phip}). 
 The forgetful morphism we aimed at is 
$$\ff_{\gamma,\ell}=\fv_1\circ\fv_2: \barM_{g,\gamma}(G)\lra \barM_{g,\gamma'}(G).
$$

Finally, one checks that the restriction of $\ff_{\gamma,\ell}$ to $\barM_{\Gamma}(G)$ factors through $\barM_{\Gamma'}(G)$. Thus
it lifts to the morphism $\ff_{\Gamma,\ell}$ desired. Also, it follows directly from the flatness of $\barM_{g,\ell}\to \barM_{g,\ell-1}$ 
that both $\ff_{\gamma,\ell}$ and $\ff_{\Gamma,\ell}$ are flat. This proves the Theorem.
\end{proof}

\subsection{Property of Witten's top Chern classes}
As $\barM_{g,\gamma}(G)^p$, we let
$\barM_\Gamma(G)^p$ be   the moduli of weakly $\Gamma$-labeled $G$-spin curves with fields.
Suppose $\Gamma$ is connected, and has $\ell$-tails and total genus $g$, then
$$\barM_\Gamma(G)^p=\barM_{\Gamma}(G)\times_{\barM_{g,\gamma}(G)}\barM_{g,\gamma}(G)^p.
$$
Following the discussion in the previous subsection, we have an induced perfect relative obstruction theory 
of $\barM_{\Gamma}(G)^p\to \barM_{\Gamma}(G)$, a $W$-induced cosection of the obstruction sheaf of 
$\barM_{\Gamma}(G)^p$ assuming all $\gamma_i$'s are narrow, and its
Witten top Chern class $[\barM_{\Gamma}(G)^p]\virt$. 

We state and prove the algebraic analogue (i.e. cycle version) of
the ``topological axioms'' stated and proved in \cite[Thm 4.1.8]{FJR2} (also cf. \cite{JKV} and \cite{Po}). 

\begin{theo}[Degeneration]\label{A}
Let $\Gamma$ be a $G$-decorated graph so that every tail $t_i$ is decorated by a narrow $\gamma_i\in G$.
Suppose $\Gamma\to \Gamma'$ is by contracting $\delta$ edges of $\Gamma$, then
$$ \tau_{\Gamma\Gamma'}^![\barM_{\Gamma}(G)^p]\virt=[\barM_{\Gamma'}(G)^p]\virt.
$$
\end{theo}

\begin{proof}
The identity follows from applying Lemma \ref{sub-cycle} to $\tau_{\Gamma\Gamma'}$.  
\end{proof}


\begin{theo}[Disjoint union of graphs]
Let the notation be as in Theorem \ref{A}.
Suppose $\Gamma$ is a disjoint union of $\Gamma_i$, $i=1,\cdots,k$.
Then
$$[\barM_\Gamma(G)^p]\virt=[\barM_{\Gamma_1}(G)^p]\virt \times\cdots\times[\barM_{\Gamma_d}(G)^p]\virt.
$$
\end{theo}

\begin{proof}
The canonical isomorphism $\barM_{\Gamma}( G)\cong \prod_{i=1}^d\barM_{\Gamma_i}(G)$ 
is induced by disjoint union of spin curves and sections, thus compatible with the construction of virtual cycles.
\end{proof}

Let the notation be as in Theorem \ref{A}. Let $e\in E(\Gamma)$ be such  that its decoration $\gamma_e: \mu_e\to G$ is narrow, and 
let $\Gamma_e$ be the graph obtained by breaking up $e$ into two tails $t_{e,-}$ and $t_{e,+}$, decorated with $\gamma_e\upmo$ and 
$\gamma_e$ (see convention after \eqref{representable}). 
In case $e$ has two distinct vertices, we define $\barM_\Gamma(G)_{e}\colon =\barM_\Gamma(G)$; otherwise, 
we form the double cover $\barM_\Gamma(G)_{e}\to \barM_\Gamma(G)$ whose closed points 
are data $(\sC, \Sigma_i^\sC, \sL_j,\hat\sC_{q+})$ so that $(\sC, \Sigma_i^\sC, \sL_j)$ are closed points in 
$\barM_\Gamma(G)$, $q\in \sC$ are nodes associated with $e\in E(\Gamma)$ and $\hat\sC_{q+}$ are
branches of the formal completions $\hat\sC_q$. We form
$$\barM_\Gamma(G)^p_{e}= \barM_\Gamma(G)^p\times_{ \barM_\Gamma(G)} \barM_\Gamma(G)_{e}.
$$
We now construct the gluing morphisms $\eta$ and $\gg$ fitting into the following two commutative squares:
\beq\label{aaa}
\begin{CD}
\barM_{\Gamma}(G)^p_e @< {\gg} << \barM_{\Gamma}(G)_e\times_{\barM_{\Gamma_e}(G)}\barM_{\Gamma_e}(G)^p
@>{}>> \barM_{\Gamma_e}(G)^p \\
@VVV @VVV @VVV\\
\barM_{\Gamma}(G)_e @= \barM_{\Gamma}(G)_e @>{\eta}>> \barM_{\Gamma_e}(G)
\end{CD}
\eeq

We construct $\eta$. By Definition \ref{4.3}, any $S$-family $\xi=(\sC,\Sigma_i^{\sC}, \sL_j,\varphi_k)$
in $\barM_\Gamma(G)(S)$ comes from gluing an $S$-family of $G$-spin curves $\ti\xi$ 
coarsely labeled by $\Gamma_{sp}$. As $(\Gamma_e)_{sp}=\Gamma_{sp}$, the family $\ti\xi$
glue to the family $\xi'=(\sC',\Sigma_i^{\sC'}, \sL'_j,\varphi'_k)$ in $\barM_{\Gamma_e}(G)(S)$. 
This transformation $\xi\mapsto \xi'$ defines the morphism $\eta$.

To define  $\gg$, we note that any $S$-family in the fiber-product consists of a triple $(\xi,\xi', (\rho_j'))$,
where $\xi\in \barM_{\Gamma}(G)_e(S)$, $\xi'$ and $\rho_j'\in \oplus_j \Gamma(\sL_j')$, etc., are as mentioned. Because $\gamma_e$ is narrow,
each section $\rho_j'$ vanishes along the marked sections (labeled) by $t_{e-}$ and $t_{e+}$.
In particulat, $\rho_j'$ glues to a global section $\rho_j\in \Gamma(\sL_j)$. The transformation
$(\xi,\xi', (\rho_j'))\mapsto (\xi, (\rho_j))$ defines the morphism $\gg$.

\begin{theo}[Compsition Law] Let the notation be as in Theorem \ref{A}; let $e\in E(\Gamma)$ be such  that its decoration 
$\gamma_e: \mu_e\to G$ is narrow, and 
let $\Gamma_e$ be as constructed. Then the morphisms  $\gg$  is an isomorphism
and $\eta$ is a $G$-torsor (cf. \eqref{aaa}). Further,
\beq\label{com}
   [\barM_{\Gamma}(G)^p_{e}]\virt = \eta^!  [\barM_{\Gamma_e}( G)^p]\virt.  
\eeq  
\end{theo}

\begin{proof}
We first prove that $\eta$ is a $G$-torsor. First, by the construction of the transformation $\xi\mapsto \xi'$,
we see that the morphism $\eta$ is \'etale and finite. We next construct
a $G$-action on $\barM_\Gamma(G)_e$. For simplicity, we assume that $e$ has two distinct vertices.
Let $\xi=(\sC,\Sigma_i^{\sC}, \sL_j,\varphi_k)\in \barM_\gamma(G)(\CC)$ and let $\eta(\xi)=(\sC',\Sigma_i^{\sC'}, \sL'_j,\varphi'_k)\in
\barM_{\Gamma_e}(G)(\CC)$. Let $f: \sC'\to\sC$ be the gluing morphism along the node $q\in \sC$; namely, $f\upmo(q)=\{q_+, q_-\}$ with 
$\hat\sC_{q+}\cong \hat\sC'_{q_+}$ under $f$.
By definition, the sheaves
$\sL_j$ and $\sL_j'$ fit into the  (gluing)  exact sequence
 \beq\label{gluing} 0\lra \sL_j\lra f\lsta \sL_j'\mapright{(a_j,b_j)} \sO_q\lra 0,
\eeq 
where $a_j$ (reps. $b_j$) factors through 
$f\lsta \sL_j'\toright{a_j'} \sL_j'\otimes_{\sO_{\sC}}\sO_{\hat \sC_{q+}}\toright{a_j^{\prime\prime}}\sO_q$ (resp. $+$ replaced by $-$).
Here we used the tautological isomorphism $\hat\sC_{q\pm}\cong \hat \sC_{q_\pm}'$.

For $c\in\Gm$, we define $a_j^{c}\defeq (c\cdot a_j^{\prime\prime})\circ a_j'$, and define
\beq\label{glueL}
\sL_j^{c}=\ker\{ (a_j^{c},b_j): f\lsta\sL'_j\lra \sO_q\}.
\eeq
By lifting to an \'etale covering of $\sC$, we see that $\sL_j^c$ is an invertible sheaf on $\sC$.

We now investigate the choices of $c=(c_1,\cdots, c_n)\in \Gm^n$ so that 
\beq\label{phic}
\varphi_k|_{\sC-q}:
\bm_k(\sL_1^{c_1}, \ldots, \sL^{c_n}_n)|_{\sC-q}\mapright{\cong}(\omega_{\sC}^{\text{log}})^{w(\bm_k)}|_{\sC-q}, \quad k=1,\cdots,n
\eeq
extend to isomorphisms over $\sC$. 
Because $\varphi_k:\bm_k(\sL_1, \ldots, \sL_n)\to(\omega_{\sC}^{\text{log}})^{w(\bm_k)}$ are isomorphisms, by the construction of 
$\sL_j^{c_j}$, we see that \eqref{phic} extend if and only if $\bm_k(\sL_1|_q,\cdots,\sL_n|_q)\equiv \bm_k(c_1\sL_1|_q,\cdots,c_n\sL_n|_q)$,
which hold if and only if $\bm_k(c_1,\cdots, c_n)=1$ for all $k$. By the definition of $G$, we conclude that \eqref{phic} extend
to isomorphisms if and only if $c\in G$.

Conversely, we  claim  that if $\bar\xi=(\sC,\Sigma_i^{\sC}, \bar\sL_j,\bar\varphi_k)\in \eta\upmo(\eta(\xi))$, 
then $\bar\sL_j=\sL_j^{c_j}$ for a $c\in G$ and $\bar\varphi_k$ are extensions of \eqref{phic}. 
Indeed, because $\eta(\xi)=\eta(\bar\xi)$, there are isomorphisms $\beta_j: \sL_j|_{\sC-q}\cong \bar\sL_j|_{\sC-q}$,
extending to $\hat \sC_{q\pm}$, so that they commute with $\varphi_k$ and $\bar\varphi_k$ for all $k$.
Therefore, there is a $c=(c_1,\cdots, c_n)\in \Gm^n$ so that $\bar\sL_j=\sL_j^{c_j}$, and 
$\beta_j$ is the identity map $\bar\sL_j|_{\sC-q}=\sL_j^{c_j}|_{\sC-q}=\sL_j|_{\sC-q}$.   Then as  $\bar\varphi_k$ are extensions of \eqref{phic},  the previous  discussion shows that  $c\in G$.   This proves that $\eta$ is a $G$-torsor.
 
Next, for any $(\xi, (\rho_j))$ in $\barM_{\Gamma}(G)^p_e$, taking $\xi'=\eta(\xi)$ and $\rho_j'=f\sta\rho_j$ defines 
$(\xi,\xi', (\rho_j'))$. This is the inverse of  $\gg$  and thus  $\gg$  is an isomorphism. This proves the first part of the Theorem.
 
We now prove the identity of virtual cycles.
Let $(\cC_\Gamma,\cL_{\Gamma,j})$ and $(\cC_{\Gamma_e},\cL_{\Gamma_e,j})$ be the universal families of $\barM_{\Gamma}( G)_e$ 
and $\barM_{\Gamma_e}(G)$, respectively. 
We form the gluing morphism $\ff$ as shown in the following commutative diagrams
\beq\label{bdia}
\begin{CD}
\cC_\Gamma @<\ff<<\cC'_{ \Gamma}\defeq \cC_{\Gamma_e}\times_{\barM_{\Gamma_e}(G)} \barM_\Gamma(G)_e
@>\eta'>> \cC_{\Gamma_e}\\
@VV{\pi_\Gamma}V@VV{ \pi'_{\Gamma}}V   @VV{\pi_{\Gamma_e}}V \\
\barM_{\Gamma}( G)_e@<{=}<< \barM_\Gamma(G)_e@>{\eta}>>  \barM_{\Gamma_e}(G).
\end{CD} \eeq
As usual, we let $\cE_{\Gamma_e}=\oplus_j \cL_{\Gamma_e,j}$ and $\cE_\Gamma=\oplus_j \cL_{\Gamma,j}$.  Let $ \sE'_{\Gamma}=\eta^{\prime\ast}\cE_{\Gamma_e}$. The right hand diagram above induces a morphism
$$ C(\pi'_{\Gamma\ast}\cE'_\Gamma)\lra C(\pi_{\Gamma_e \ast}\cE_{\Gamma_e})=\barM_{\Gamma_e}(G)^p$$
 compatible with relative obstruction theories. By \'etaleness of $\eta$ and Lemma \ref{sub-cycle}  
\beq\label{form1}\eta^![\barM_{\Gamma_e}(G)^p]\virt=[C(\pi'_{\Gamma\ast}\sE'_\Gamma)]\virt.\eeq


On the other hand, the  gluing morphism $\ff$  induces a canonical  
$\ff\sta\cE_\Gamma\cong \cE'_\Gamma$ and
\beq\label{row}  \barM_{\Gamma}(G)_e^p= C(\pi_{\Gamma\ast}\cE_\Gamma)\mapright{\cong}  
C(\pi'_{\Gamma\ast}\cE'_\Gamma) ,
\eeq
which is the inverse of  $\gg$. Taking cohomology of \eqref{gluing} and apply  Lemma \ref{hi}'s proof,  
one verifies that (\ref{row})   identifies deformation theories and cosections.  Thus it induces
\beq\label{iso1}
[\barM_{\Gamma}( G)_e^p]\virt =[C(\pi'_{\Gamma\ast}\sE'_\Gamma)]\virt\in A\lsta(\barM_\Gamma(G)_e)
\eeq
Together with (\ref{form1}) this proves the Theorem.
\end{proof}

We state and prove the forgetting tails Theorem.


\begin{theo}[Forgetting tails]
Suppose $\Gamma$ has its last tail $t_\ell$ decorated by $\gamma_\ell: \mu_d\to G$ via $\zeta_d\mapsto \jmath_\bd$ (cf. Subsection 2.1). 
Let $\Gamma'$ be the stabilization of the graph after removing $t_\ell$ from $\Gamma$, and let
$$\ff_{\Gamma,\ell}:\barM_\Gamma(G)\lra \barM_{\Gamma'}(G)
$$
be the forgetful morphism constructed in Theorem \ref{fgtail}. 
Then $\ff_{\Gamma,\ell}$ is flat and satisfies $[\barM_\Gamma(G)^p]\virt=(\ff_{\Gamma,\ell})\sta [\barM_{\Gamma'}(G)]\virt.$
\end{theo}

\begin{proof}
Let
$\pi_\Gamma:\sC_\Gamma\to \barM_{\Gamma}( G)$ and $\pi_{\Gamma'}:\sC_{\Gamma'}\to \barM_{\Gamma'}(G)$
be the universal curves; let
$$\fu_1: \sC_{\Gamma}\lra \sC_{\Gamma'/\Gamma}\defeq  \sC_{\Gamma'}\times_{\barM_{\Gamma'}(G)} \barM_{\Gamma}( G)
$$
be the tautological morphism induced by the construction of the morphisms  $\fv_1$ and $\fv_2$  in the proof of
Theorem \ref{fgtail}.  
We let $\fu_2:\sC_{\Gamma'/\Gamma}\to \sC_{\Gamma'}$ be the projection.

 
For $\{\sL_{j}\}_{j=1}^n$  and $\{\sL'_{j}\}_{j=1}^n$ being the universal invertible sheaves on $\sC_\Gamma$ and $\sC_{\Gamma'}$,
by the construction of $\fv_1$ and $\fv_2$,  we have $\fu_2\sta \sL'_j=\fu_{1\ast}\sL_j$.
Because of the assumption on $\gamma_\ell$,  one has
\beq\label{inc}(\fu_2\circ\fu_1)\sta\sL'_{j}=\fu_1\sta\fu_{1\ast}\sL_j \cong \sL_{j}(-\delta_j\Sigma_\ell^{\sC_\Gamma})\mapright{\sub}\sL_j.
\eeq
Denoting $\sE' =\oplus_j\sL'_{j}$ and $\sE =\oplus_j\sL_{j}$, and applying Lemma \ref{hi} to $(\fu_2\circ\fu_1)\sta\sE'\sub \sE $,
we obtain $R^\bullet \pi_{\Gamma\ast}(\fu_2\circ\fu_1)\sta\sE'\cong R^\bullet \pi_{\Gamma\ast} \sE$, and the induced isomorphism 
$$C(\pi_{\Gamma\ast}\sE) \mapright{\cong} C(\pi_{\Gamma'\ast}\sE')\times_{\barM_{\Gamma'}( G)}\barM_{\Gamma}( G).$$
Since the isomorphism is compatible with deformation theories and cosections,
$$[\barM_\Gamma(G)^p]\virt=[C(\pi_{\Gamma\ast}\sE)]\virt=\ff_{\Gamma,\ell}\sta [C(\pi_{\Gamma'\ast}\sE')]\virt=
\ff_{\Gamma,\ell}\sta [\barM_{\Gamma'}(G)]\virt.
$$
This proves the Theorem. 
\end{proof}

Let $([\CC^n/G],W)=([\CC^{n_1}\times\CC^{n_2}\!/G_1\times G_2], W_1+W_2)$ 
be the product of two LG spaces $([\CC^{n_i}/G_i],W_i)$.
Let $\gamma^i\in (G_i)^\ell$ be $g$-admissible and narrow, and let $\gamma\defeq \gamma^1\times \gamma^2$ be the direct product
$\gamma_i=\gamma^1_i\times \gamma^2_i$ for $1\le i\le \ell$. 
We let $\Gamma$ be a $G$-decorated graph with $\ell$ ordered tails decorated by $\gamma_1,\cdots, \gamma_\ell$, respectively.
We let $\Gamma_j$ be $\Gamma$ with the decorations of its $i$-th tails replaced by $\gamma^j_i$, for all $i$.

\begin{theo}[Product of LG spaces]
Let the notation be as stated. Then we have the partial forgetting morphism
$f_j: \barM_{\Gamma}(G_1\times G_2) \lra \barM_{\Gamma_j}(G_j)$. 
Further, 
$$(f_1\times f_2)\sta([\barM_{\Gamma_1}(G_1)^p]\virt\times  [\barM_{\Gamma_2}(G_2)^p]\virt )=
[\barM_{\Gamma}(G_1\times G_2)^p]\virt.
$$
\end{theo}

\begin{proof}

We first show that 
$$f_1\times f_2: \barM_{\Gamma}(G_1\times G_2)\to 
\barM_{\Gamma_1}(G_1)\times \barM_{\Gamma_2}(G_2)
$$ 
is an l.c.i. morphism. Using the \'etale forgetful morphisms from $\barM_{\Gamma_1}(G_1)$, $\barM_{\Gamma_2}(G_2)$ and  $\barM_{\Gamma}(G_1\times G_2)$ 
to $\mathfrak M_{g, \ell}^{\rm tw}$, 
$f_1\times f_2$ is the composite
$$\barM_{\Gamma}(G_1\times G_2)\lra \barM_{\Gamma_1}(G_1)\times_{\mathfrak M_{g, \ell}^{\rm tw}} \barM_{\Gamma_2}(G_2)
\lra \barM_{\Gamma_1}(G_1)\times \barM_{\Gamma_2}(G_2).
$$
By construction, the first arrow is an open embedding; the second arrow is a lift of the diagonal morphism
$\Delta:  \mathfrak M_{g, \ell}^{\rm tw}\to \mathfrak M_{g, \ell}^{\rm tw}\times \mathfrak M_{g, \ell}^{\rm tw}$. 
Since $\mathfrak M_{g, \ell}^{\rm tw}$ is smooth, $f_1\times f_2$ is an l.c.i. morphism.

The remainder proof is parallel to the prior proofs, and will be omitted.
\end{proof}


\subsection{Free cases}

Let $S$ be a pure dimensional proper DM stack and $\iota: S\to M$ be a closed embedding.
We form $Y=X\times_M S$, with $\lam:Y\to S$ the projection. The relative obstruction theory of $X\to M$ pulls back to that
of $Y\to S$, thus $\Ob_{Y/S}=\iota\sta\Ob_{X/M}$, and the cosection $\sigma: \Ob_{X/M}\to\sO_X$ pulls back to a cosection
$\sigma_S: \Ob_{Y/S}\to\sO_S$, which has all the properties required to define the cosection localized virtual cycle
$$[Y]\virt=0^!_{\sigma_S,loc}[\bC_{Y/S}]\in A\lsta (D(\sigma_S)).
$$

\begin{defi}
We call $\iota$ free if all $H^i(\iota\sta \Rpi_{M\ast}\cL_{M,j})$ are locally free sheaves of
$\sO_S$-modules. 
\end{defi}

In this subsection, we derive an explicit formula of $[Y]\virt$ in case $\iota$ is free.

We first generalize a lemma on Segre class to the weighted projective bundle case.
Let $S$ be a proper integral DM stack over $\CC$. Let $F_1,\cdots, F_n$ be vector bundles on
$S$ of rank $r_1,\cdots, r_n$. We denote $F=\oplus_{i=1}^n F_i$, $r=\rank F$, and let $1_S$ be the trivial line bundle on $S$.
For relatively prime positive integers $(e_1,\cdots, e_n)$, we introduce a $\Gm$-action on $F\oplus 1_S$ via
\beq\label{weight}
(v_1,\cdots,v_n,c)^t=(t^{e_1}v_1,\cdots,t^{e_n}v_n,t\cdot c).
\eeq
Let $0_S\sub F\oplus 1_S$ be the zero section. We define
\beq\label{ZZ}
Z=(F\oplus 1_S-0_S)/\Gm.
\eeq
The total space of $F$ embeds in $Z$ via $F\to F\oplus 1_S$: $v\mapsto [v,1]$.
Let $D=Z-F$ be its complement, which is a divisor in $Z$.
Let $\pi:Z\to S$ be the projection. For a vector bundle $E$, we denote $c(E)(t)=\sum c_i(E)t^i$, the total Chern polynomial 
of $E$.

\begin{lemm}\label{410} We have an identity of cycles
$$\sum_i \pi\lsta\bl c_1(\sO_Z(D))^i\cap [Z])\cdot t^{i}= t^r\bl \prod_{j=1}^n (e_j)^{r_j}c(F_j)(t/e_j)\br^{-1}.
$$
\end{lemm}

\begin{proof}  First using projection formula, one sees that if the Lemma holds for vector bundles over $S'$, and there is 
a generic finite proper morphism $S'\to S$, then the Lemma holds for vector bundles over $S$ as stated. 
By Chow's Lemma, we can find a generic finite morphism from an integral projective scheme to $S$, thus to prove the Lemma
we only need to treat the case where $S$ is projective and integral. Next, by splitting principle, we can assume all $F_j$ are direct sum of line bundles.
Thus without lose of generality, we can assume all $F_j$ are line bundles on $S$.
Applying the covering trick (\cite[Lemma 2.1]{GieS}), after passing to an integral projective scheme generically finite 
over $S$, we can assume that there are line bundles $L_{j}$ on $S$ so that $ L_{j}^{\otimes e_j}\cong F_{j}$ for every $j$.

We let  
$L=\oplus_j L_j$,  $Z'=\PP(L\oplus 1_S)$ with projection $\pi': Z'\to S$, and 
  $\rho: Z'\to Z$ be
induced by the map
$$L \oplus 1_S\ni ((u_1,\cdots,u_n),c)\mapsto 
((u_1^{e_1},\cdots,u_n^{e_n}),c)\in F\oplus 1_S.
$$
It is a flat, finite $S$-morphism of degree $\bar e=e_1^{r_1}\cdots e_n^{r_n}$. 
Let $ D'\sub  Z'$ be the divisor at infinity (defined by $c=0$). Then $\rho\sta\sO(D)=\sO( D')$ and $\pi\circ\rho=\pi'$.
Applying the projection formula, we obtain
$$\sum_i\pi'\lsta\bl c_1(\sO( D'))^i\br\cdot t^i=\sum_i\pi\lsta\rho\lsta\rho\sta\bl c_1(\sO(D))^i\br\cdot t^i=
e\sum_i\pi\lsta\bl c_1(\sO(D))^i\br t^i.$$
On the other hand, since $Z'$ is a projective bundle over $S$, 
we have $\bl c(L)(t)\br\upmo=\sum_i\pi'\lsta\bl c_1(\sO( D'))^i\br\cdot t^{i-r}$. 
Therefore, 
$$\sum_i\pi\lsta\bl c_1(\sO(D))^i\br t^{i-r}=(\bar e)^{-1}  \bl c(L)(t)\br^{-1}=(\bar e)^{-1} \Bigl(\prod_{j=1}^n c(F_j)(t/\delta_j)\Bigr)^{-1}.
$$
Here the second identity is from  $F=\oplus_{j} F_{j}=\oplus_j L_{j}^{\otimes \delta_j}$. This proves the Lemma.
 \end{proof}

Back to the situation introduced at the beginning of this subsection, 
and   suppose $\iota$ is free. 
As usual, we let $\pi_S: \cC_S\to S$ be $\cC_M\times_M S\to S$, and  $\cL_{S,j}$ be the pull-back of $\cL_{M,j}$.
We let ${F}_j=R^0\pi_{S\ast} \cL_{S,j}$, ${F}=\oplus_j{F}_j$, $G_j=R^1\pi_{S\ast} \cL_{S,j}$, and $G=\oplus _j G_j$.
Since $\iota$ is free, they are locally free sheaves (vector bundles) on $S$. 
Because pull-back of the relative obstruction sheaf of $X/M$ to $Y$ is the relative obstruction sheaf of $Y/S$, we have $\Ob_{Y/S}= \lam\sta G$.

As $G$ is locally free, by base change property, we know that $Y$ is the total space of $F$, as a stack over $S$.
In particular, $Y$ is smooth over $S$. Thus the intrinsic normal cone is the zero section of the bundle $\lam\sta G$.
Therefore, using the cosection $\sigma|_S: \Ob_{Y/S}\to\sO_S$, the cosection localized virtual cycle of $Y$ is
$$[Y]\virt=0^!_{\sigma|_S, loc}[0_{\lam\sta G}]\in A\lsta (S).
$$
In particular, if $S_k$'s are the integral components of $S$ with multiplicity $a_k$, then $[Y]\virt=\sum_k a_k [T\times_S S_k]\virt$.
Therefore, it suffices to treat the case where $S$ is integral, which we suppose from now on.  

We now identity the cosection $\sigma|_S: \lam\sta G\to\sO_Y$.
Since $Y$ is integral and $\lam\sta G$ is locally free, we only need to identify $\sigma(y)=\sigma|_y$ for any
closed point $y\in Y$. 
Recall $W_a(x)=x_1^{m_{a1}}\cdots x_n^{m_{an}}$ are monomials of $W(x)$;
$W_a(x)_j=\partial_{x_j} W_a(x)$ and $\ti W_a(x)_j=(m_{aj})\upmo W_a(x)_j$.
Following the construction of \eqref{si-cosection}, we let $\ti y=(\ti y_j)\in \oplus_j\lam\sta F_j|_y=\lam\sta F|_y$
be the point $y\in Y$ under the identification $Y=\lam\sta F$, the total space of $\lam\sta F$. 
For any $\dot \rho_j\in \lam\sta G_j|_y$, we introduce
$$\ti W_a(\ti y)_j\defeq \ti W_a(\ti y_{1},\cdots, \ti y_{n})_j\in \omega\ulog_{\cC_M/M}\otimes \cL_{M,j}\upmo|_{\lam(y)},
$$
and for $\alpha_a$ the coefficients of $W_a$ in $W$, we have
$$\sigma(y) = (\sigma_j(y)): \oplus_{j=1}^n \lam\sta G_j|_y\lra \CC, \quad \sigma_j(y)(\dot \rho_j)=
\bl \sum_a \alpha_a\cdot m_{aj}\cdot \ti W_a(\ti y)_j \br\cdot\dot\rho_j.
$$
Here we have used that $\gamma$ is narrow. 

\begin{prop}\label{610}
We let $r_j=\rank F_j$, $r=\sum_j r_j$, $s_j=\rank G_j$, and $s=\sum_j s_j$. We set $\varepsilon_j=\delta_j-d$. Then
\beq\label{Yvirt}[Y]\virt=\mathrm{Coeff}_{t^{s-r}} \left(\frac{\prod_{j=1}^n\varepsilon_j^{s_j}c(G_j)(t/\varepsilon_j)}  { \prod_{j=1}^n \delta_j^{r_j}c(F_j)(t/\delta_j)}\right). 
\eeq
 \end{prop} 

\begin{proof}  
For the vector bundle $F_j$($=R\pi_{S\ast}\sL_{S,j}$) introduced, we form the quotient
$Z$ as in \eqref{ZZ} with $e_j=\delta_j$. We let $\pi: Z\to S$ be the projection. As $Y$ is the total space of
$F=\oplus_j F_j$, $Y\sub Z$ is open with complement the divisor $D=(c=0)$, and $\pi|_Y=\lam$.
We let $V_j=\pi\sta G_j$, and let $V=\oplus_j V_j$; thus $\sigma_j: V_j\to\sO_Y$ and $\sigma: V\to \sO_Y$.
  
Consider the duals $\sigma_j\dual\in\Gamma(Y,V_j\dual)$ and $\sigma\dual\in\Gamma(Y,V\dual)$. 
The zero locus $(\sigma\dual=0)$ has reduced part equal to $D(\sigma)$.
Hence by Lemma \ref{Dege}, it is supported on $S\sub Y$ (as the zero section of $F$).

Since $\ti W_a(x)_j$ is quasi-homogeneous of total degree $-\varepsilon_j=d-\delta_j$,
each homomorphism $\sigma_j$ extends uniquely over $Z$ to
$\bar\sigma_j:V_j\otimes\sO_Z(\varepsilon_jD)\lra \sO_{Z}$
such that the degeneracy (non-surective) loci of $\bar\sigma:=\oplus_j\bar\sigma_j$ is supported on $S$. 
Thus by denoting $\ti V=\oplus_j V_j\otimes\sO_Z(\varepsilon_jD)$, the section $\bar s:=\bar\sigma\dual\in \Gamma(Y,\ti V\dual)$
has zero locus supported on $S$ and  $(Z,\ti V\dual,\bar \sigma\dual)$ is an extension of $(Y,V\dual|_Y,\sigma\dual)$.

Since $Y/S$ is smooth, the cone $\bC_{Y/S}\sub h^1/h^0(E\bul_{Y/S})$ is the zero section. 
By our construction of Witten's top Chern class, we have
$$0^!_{\sigma,loc}[\bC_{Y/S}]=(-1)^{s} e_{\sigma\dual,loc}(V\dual)
=e_{\bar \sigma\dual,loc}\ti V\dual =\pi\lsta e(\ti V\dual)=(-1)^{s} \pi\lsta e(\ti V) \in A\lsta(S).
$$
Here, the second identity holds because $\bar\sigma\dual$ is non-vanishing along $D=Z-Y$;
the third identity holds because $Z$ is proper.
Hence the Witten's top Chern class is
\beq\label{e0}\pi\lsta\Bigl( e(\oplus_j V_j(\varepsilon_j D))\Bigr)=\pi\lsta\Bigl( 
\prod_j e(\pi^\ast G_j\otimes\sO_Z(\varepsilon_j D))\Bigr). \eeq
Denoting $\sO(1)=\sO_Z(D)$ and letting
$I=\{(i_1,\cdots,i_n)\in \ZZ^n \mid 0\leq i_j\leq s_j\}$, we have
\begin{eqnarray*}
&&\pi\lsta\Bigl(\prod_{j=1}^n \sum_{i_j=0}^{s_j}c_{s_j-i_j}(\pi\sta G_j)\cdot c_1(\sO(\varepsilon_j))^{i_j}\Bigr)\\
&=& \pi\lsta\Bigl(\sum_{(i_1,\cdots,i_n)\in I}\bl \prod_{j=1}^n c_{s_j-i_j}(\pi\sta G_j)\br\bl c_1(\sO(1))^{i_1+\cdots+i_n}\br 
(\varepsilon_1^{i_1}\varepsilon_2^{i_2}\cdots\varepsilon_n^{i_n})\Bigr)\\
&=&\sum_{(i_1,\cdots,i_n)\in I}\bl \prod_{j=1}^n \varepsilon_j^{s_j}c_{s_j-i_j}(G_j)\varepsilon_j^{-(s_j-i_j)}\br \cdot \pi\lsta [c_1(\sO(1))]^{i_1+\cdots+i_n}\\
&=&\text{Coeff}_{t^{s-r}}\Bigl(\big(\prod_{j=1}^n\varepsilon_j^{s_j}c(G_j)(t/\varepsilon_j)\big)\cdot 
\sum_i\pi\lsta\bl c_1(\sO(1))^i\br t^{i-r}\Bigr).
\end{eqnarray*}
Applying Lemma \ref{410}, we obtain (\ref{Yvirt}).
\end{proof}

 \begin{coro} In the case $n=1$, $\delta_1=1$, and let $\varepsilon=1-d$,   
$$[Y]\virt= \mathrm{Coeff}_{t^{s-r}}\bigl(\varepsilon^{h}c(G)(t/\varepsilon) /c(F)(t)\bigr).
$$
\end{coro}

The following Propositions are the ``Concavity" and ``Index zero" axioms stated and proved in \cite[Thm 4.1.8]{FJR2}.
 
\begin{prop}\label{Concavity} In case $r=0$,  we have  $[Y]\virt=c_{top}(G).$
\end{prop}

\begin{prop}\label{Constant-dimension}
If $r=s$, we have
$$[Y]\virt= \frac{(\varepsilon_1^{s_1}\cdots \varepsilon_n^{s_n})\prod_j s_j}{(\delta_1^{r_1}\cdots \delta_n^{r_n}) \prod_j r_j}\cdot [S] \in A_{\dim S}(S).$$
\end{prop}

\begin{proof} The lowest-degree terms of the two power series  $\prod_{j=1}^n\varepsilon_j^{s_j}c(G_j)({t}/{\varepsilon_j})$ and  
$\prod_{j=1}^n \delta_j^{r_j}c(F_j)({t}/{\delta_j})$ are constant terms $(\varepsilon_1^{s_1}\cdots \varepsilon_n^{s_n})\prod_j s_j$ 
and $(\delta_1^{r_1}\cdots \delta_n^{r_n}) \prod_j r_j $, respectively. As $r-s=0$, (\ref{Yvirt}) is
the ratio of these two constant terms; the ratio is the ``degree of Witten map" in \cite[Thm 4.1.8 (5)(b)]{FJR2}.\end{proof}

\section{Comparison with other constructions}

In this section, we will prove the equivalences between our construction with
the previously known constructions of Witten's top Chern classes.

We fix an LG space $([\CC^n/G],W)$, integers $g$, $\ell$ and a narrow $\gamma=(\gamma_1,\cdots,\gamma_\ell)$ 
of faithful cyclic representations in $G$. We abbreviate $M=\barM_{g,\gamma}(G)$ and $X=\barM_{g,\gamma}(G)^p$,
and let $\pi_M:\cC_M\to M$ and $\cL_1,\cdots,\cL_n$ be (part of) the universal family of
$M$.  We denote $\cE=\cL_1\oplus \cdots \oplus \cL_n$.

\subsection{Comparison with Fan-Javis-Ruan's construction} 

In this subsection, we prove that the associated homology class of $[\barM_{g,\gamma}(G)^p]\virt$ coincides
with the Witten's top Chern class constructed by Fan-Jarvis-Ruan in (\cite[Thm 4.1.8 ]{FJR2}). 

We begin with an explicit description of $X\to M$ using complex representatives of $\Rpi_{M\ast} \cE$.
Let $F\bul=[\zeta:F^0\to F^1]$ be a two-term complex of locally free sheaves of $\sO_M$-modules so that
$F\bul\cong R\bul\pi_{M\ast}\cE$ as derived objects. 
We let $Y$ be the total space of the associated vector bundle of $F^0$; let
$\ti q:Y\to M$ be the projection and let $V=\ti q\sta F^1$. The homomorphism $\zeta: F^0\to F^1$ induces a tautological
section $\bar\zeta\in \Gamma(Y,V)$ 
whose vanishing locus $\bar\zeta\upmo(0)\sub Y$, by the isomorphism $\Rpi_{M\ast} \cE\cong F\bul$, 
is canonically isomorphic to $X$. 
In this way, we view $X$ as a substack of $Y$ defined by the vanishing of $\bar\zeta$, and write $X=(\bar\zeta=0)$.
Let $q=\ti q|_X:X\to M$ be the projection.

By the construction $X=(\bar\zeta=0)$, we have identity of complexes
$$[d\bar\zeta|_X: \Omega\dual_{Y/M}|_X\to V|_X]=  [q\sta\zeta: q\sta F^0\to q\sta F^1]= q\sta F\bul.
$$
Combined with (\ref{XY}), and denoting $E_\zeta=q\sta F\bul$, we have induced isomorphism
\beq\label{6.7}
E_{X/M}\bul= q\sta (\Rpi_{M\ast} \cE)\cong q\sta F\bul =E_\zeta\bul.
\eeq
Let $I_X$ be the ideal sheaf of $X\sub Y$. That $X=\bar\zeta\upmo(0)$ gives us the (truncated) perfect obstruction theory $\phi_{\zeta}$:
\beq\label{XE}
\phi_\zeta\dual: [V\dual|_X\to \Omega_{Y/M}|_X]= (E_\zeta\bul)\dual\lra
\tLL_{X/M}=[I_X/I_X^2\to \Omega_{Y/M}|_X].
\eeq

We let $\bC_\zeta\sub h^1/h^0(E_\zeta\bul)$ be the virtual normal cone of $X$ via the obstruction theory $\phi_\zeta$.
Using the isomorphism \eqref{6.7}, $\bC_\zeta$ is a closed substack of $h^1/h^0(E_{X/M}\bul)\cong h^1/h^0(E_\zeta\bul)$.

\begin{lemm}\label{inde}
The cycle $[\bC_\zeta]\in Z\lsta (h^1/h^0(E\bul_{X/M}))$ is independent of the choice of the complexes $F\bul$ satisfying 
$F\bul\cong \Rpi_{M\ast}\cE$.
 \end{lemm}

\begin{proof}
Let $F\bul$ and $\ti F\bul$ be two two terms complexes of locally free sheaves so that $F\bul\cong \ti F\bul\cong \Rpi_{M\ast}\cE$.
Because derived isomorphism of complexes are composition of quasi-isomorphic complex homomorphisms,
to prove the Lemma we only need to prove the case where $F\bul\cong \ti F\bul$ is induced by a
complex homomrophism $f: F\bul\to \ti F\bul$. In this case, the conclusion is straightforward.
\end{proof}
\vsp

We now describe Fan-Javis-Ruan's construction of Witten's top Chern class for the narrow $\gamma$.
As $\gamma$ is narrow, the isomorphism $\Phi_a$ in \eqref{uni} induces a sheaf homomorphism
\beq\label{taua}\tau_a:W_a(\cL_1,\cdots,\cL_n)\lra \omega_{\cC_M/M}.
\eeq
We choose complexes of locally free sheaves of $\sO_{M}$-modules
$\sF\bul_j=[\sF^0_j\to \sF^1_j]$ such that $\sF\bul_j\cong \Rpi_{M\ast} \cL_j$ as derived objects. 
Using (tensor) product of complexes of locally free sheaves, we can
substituting $x_j$ in the monomial $W_a(x)=x_1^{m_{a1}}\cdots x_n^{m_{an}}$
by $\sF\bul_j$ to obtain the complex
$$W_a(\sF\bul)\defeq (\sF\bul_1)^{\otimes m_{a1}}\times \cdots \times(\sF\bul_n)^{\otimes m_{an}}.
$$

\begin{lemm}\label{lem6.2}
We can find $\sF\bul_j=[\zeta_i:\sF^0_j\to \sF^1_j]$, $j=1,\cdots, n$,
so that the derived morphism $W_a(\sF\bul)\to \Rpi\lsta \omega_{\cC_M/M}$
induced by $\tau_a$'s in \eqref{taua}
are realized by homomorphisms of complexes
$$\nu_a: W_a(\sF\bul)\lra \Rpi_{M\ast}   \omega_{\cC_M/M}\cong \sO_M[-1].
$$
\end{lemm}

\begin{proof} It is proved in \cite{PV2}. 
We point out that the $\Sigma_j$ and $\Sigma_M$ defined in  \cite[(4.8)]{PV2} collect marked points which are not 
narrow and hence are empty set in our case. Thus the homomorphism in 
\cite[Lemma 4.2.2]{PV2} is the same as  (\ref{taua}), and are identical to $t_M=\tau_M$ in diagram \cite[(4.14)]{PV2}. 
By the argument following \cite[Lemma 4.2.5]{PV2}, two-term perfect  resolution $\sF\bul_j$ realizing $\tau_M$ (our $\nu_a$) 
as complex homomorphism exists. Since the existence of $\nu_a$ only relies on vanishings of cohomology groups, 
we can choose $\sF\bul_j$ so that all $\nu_a$ are realized as complex homomorphisms.
\end{proof}

We fix the complexes $\sF\bul_j$ and the complex homomorphisms $\nu_a$ given by the Lemma \ref{lem6.2}.   
For the  given $W_a(x)$, we continue to denote by $W_a(x)_j$ the 
partial derivative $\partial_{x_j} W_a(x)$, and denote
$\ti W_a(x)_j=(m_{aj})\upmo W_a(x)_j$(a monomial of coefficient $1$).
We abbreviate $\ti W_a(\sF^0)_j=\ti W_a(\sF^0_1,\cdots, \sF^0_n)_j$.

Since the degree one term of the complex 
$W_a(\sF\bul)$ is 
$\oplus_{j=1}^n  \ti W_a(\sF^0)_j\otimes \sF^1_j$, the degree one part of the $\nu_a$ is
\beq\label{nu1}
\nu_{a}^1:\oplus_{j=1}^n  \ti W_a(\sF^0)_j\otimes \sF^1_j\lra \sO_M.
\eeq
We let $F\bul=\oplus_{j=1}^n \sF\bul_j$; namely, $F^i=\sF^i_1\oplus\cdots \oplus \sF^i_n$ and
$F\bul=[\zeta: F^0\to F^1]$ is given by
$\zeta=\oplus_{j=1}^n \zeta_j$. For $\cE=\oplus_{j=1}^n \cL_j$,
we have $F\bul\cong \Rpi_{M\ast} \cE$ as derived objects.
Following the notation developed before \eqref{6.7}, we let $(Y, V, \bar\zeta)$ be constructed from $F\bul$, which gives us
the isomorphism
${\bar\zeta}\upmo(0)= X$, the obstruction theory $\phi_\zeta$, and the deformation complex
$E_\zeta\bul=q\sta F\bul$. 

We now show that each $\nu_a^1$ defines a homomorphism $\eta_a: V\to\sO_Y$. Indeed,
let $S$ be an affine scheme and $\rho: S\to Y$ be a morphism; let $\rho'=q\circ\rho: S\to M$
be the composite. As $Y$ is the total space of $F^0$, $\rho$ is given by a section
$s=(s_1,\cdots,s_n)\in \oplus_{j=1}^n \Gamma(\rho^{\prime\ast}\sF^0_j)$.
We abbreviate $W_a(s)_j=W_a(s_1,\cdots, s_n)_j$. We define
\beq\label{etaa}\eta_{a,\rho}: \rho\sta V=\oplus_{j=1}^n \rho^{\prime\ast} \sF^1_j\lra \sO_S
\eeq
via $(\dot s_1,\cdots,\dot s_n)\mapsto \rho^{\prime\ast}(\nu_a^1)(W_a(s)_1 \dot s_1,\cdots, W_a(s)_n \dot s_n)$.
As this definition satisfies base change, such $\eta_{a,\rho}$ descents to a homomorphism
$\eta_a: V\to \sO_Y$. Let $\alpha_a$ be the coefficient of $W_a$ in $W$. We define
$$\eta=\sum_a \alpha_a \eta_a: V\lra \sO_Y.
$$

\begin{lemm}\label{vani} We have $\eta\circ {\bar\zeta}=0$.
\end{lemm}

\begin{proof}
It suffices to show that
for the $\rho: S\to Y$ as before, $\rho\sta(\eta\circ{\bar\zeta})=0$.
To this end, we give another interpretation of this composition. 

First, by the definition of symmetric product of complexes, the part of 
$W_a(\sF\bul)$ containing the
degree zero and one terms is
$$\delta^0_a: W_a(\sF^0)\lra\oplus_{j=1}^n  \ti W_a(\sF^0)_j\otimes \sF^1_j,
$$
whose pullback via $\rho'$ is 
\beq\label{2}\rho^{\prime\ast}(\delta_a^0)(W_a(e))=\sum W_a(e)_j\otimes \zeta_j(e_j),
\quad e=(e_1,\cdots,e_n)\in \oplus_{j=1}^n  \rho^{\prime\ast}\sF^0_j.
\eeq
Because $\nu_a$ is a homomorphism of complexes, we have $\nu_a^1\circ \delta_a^0=0$.
Therefore
$$\rho\sta(\eta\circ{\bar\zeta})=\rho\sta\eta\circ\rho\sta{\bar\zeta}=\rho^{\prime\ast}(\nu_a^1)
(W_a(s)_j\zeta_j(s_j)) 
=(\rho^{\prime\ast}(\nu_a^1)\circ\rho^{\prime\ast})(\delta_a^0)(W_a(s))
=0.
$$
This proves the Lemma.
\end{proof}

%
%
%
%
%

We next endow $F_1$ a hermitian metric, thus inducing a hermitian metric  $(\cdot,\cdot)$ on $V$.
Via the linear-antilinear pairing $(\cdot, \cdot): V\times V\to \CC$, we obtain an anti-$\CC$-linear isomorphism $c: V\dual\to V$.
Viewing $\eta$ as a section in $V\dual$, then $c(\eta)\in C^\infty(V)$. We define
$$d_\eta=\bar\zeta+c(\eta)\in C^\infty(V).
$$

\begin{lemm}\label{corozero} 
The vanishing locus of $d_\eta$ is the zero section of $Y\to M$, thus is isomorphic to $M$ and proper.
\end{lemm}

\begin{proof} For a closed point $v\in Y$ lying over $w\in M$ such that
$d_\eta(v)=\bar\zeta(v)+c(\eta)(v)=0$, by composing  with $\eta$,  we have $\eta\circ c(\eta)(v)=0$.
As $c$ is induced by the duality from a hemitian metric,  $\eta\circ c(\eta)(v)=0$ if and only if
$\eta(v)=0$.
Applying Lemma \ref{Dege}, we conclude that this is possible if and only if $v$ lies in the zero section of $Y\to M$.
\end{proof}
 
By the compactness of the vanishing locus of the smooth section $d_\eta$, the pair
$(V,d_\eta)$ defines a localized Euler class
$$e\loc(V,d_\eta)\in H\lsta(M,\QQ).
$$
In \cite[Thm 4.1.8 (5)]{FJR2}, the authors constructed the Witten's top Chern class
$[\fW_{g}(\gamma)]\virt_{FJRW}$ (analogue to the class $[\barM_{g,\gamma}(G)^p]\virt$, where $W$ stands for the
polynomial $W$,) in a more
general setting; in the case $\gamma$ is narrow,
they proved that the class $[\fW_{g}(\gamma)]\virt_{FJRW}$ is given by the $e\loc(V,d_\eta)$
just constructed. We state it as a Proposition.

\begin{prop}[{\cite[Thm 4.1.8 (5)]{FJR2}}]\label{axiom5}
Let $([\CC^n/G],W)$ be an LG space; let $g$, $\ell$ be integers and $\gamma=(\gamma_i)_{i=1}^\ell$ be narrow.
Let $\epsilon=\rank F^0-\rank F^1$. Then
$$[\fW_{g}(\gamma)]\virt_{FJRW}=(-1)^\epsilon\cdot e\loc(V,d_\eta)
\in H\lsta(M).
$$
\end{prop}

Let $\imath\lsta: A\lsta( M)\to H\lsta (M,\QQ)$ be the tautological homomorphism.

\begin{theo}\label{cop-FJRW}
Let the notation be as in Proposition \ref{axiom5}. Then
$$\imath\lsta[\barM_{g,\gamma}(G)^p]\virt=(-1)^\epsilon\cdot [\fW_g(\gamma)]\virt_{FJRW}\in H\lsta (M,\QQ).
$$
\end{theo}

Here is our strategy to prove this Theorem. Using the quasi-homogeneous polynomial $W$ we will 
construct a cosection $\sigma_\eta$ of the obstruction sheaf of 
the obstruction theory $\phi_\zeta$, with the property that its degeneracy loci is the zero section of $Y\to M$.
In this way, we will obtain a cosection localized virtual cycle $[X]\virt_{\sigma_\eta}$. Using the analytic construction of
cosection localized virtual cycle \cite[Appendix]{KL}, we conclude that $e\loc(V,d_\eta)=\imath\lsta [X]\virt_{\sigma_\eta}$.
Finally, by comparing the obstruction theories, we prove $[\barM_{g,\gamma}(G)^p]\virt=[X]\virt_{\sigma_\eta}$.

Now we construct the cosection $\sigma_\eta$. 
Let $\eta_X=\eta|_X$. Since $\eta\circ\bar\zeta=0$ and $\bar\zeta|_X=0$, for the differential $d\bar\zeta|_X: T_X Y\to V|_X$,  
we have $\eta_X\circ d\bar\zeta|_X=0$. Thus $\eta_X$ lifts to a cosection 
\beq\label{si-eta}
\sigma_\eta: \Ob_{X/M}=H^1(E\bul_\zeta) \lra \sO_X
\eeq
whose composition with $q\sta \Omega\dual_M\to H^1(E\bul_\zeta)$ vanishes.

By the proof of Lemma \ref{corozero}, we know that the $\sigma_\eta$ is surjective away from the zero section of $V\to M$;
thus the degeneracy loci $D(\sigma_\eta)=M\sub X$. This gives us the associated cosection localized virtual cycle 
$[X]\virt_{\sigma_\eta}\in A\lsta (M)$.

\begin{lemm}\label{lem6.6}
Let the notation be as stated, then we have
$\imath\lsta[X]\virt_{\sigma_\eta}=e\loc(V,d_\eta)$.
\end{lemm}

\begin{proof}
First, by the construction of the smooth section $d_\eta$, the $C^\infty$-homomorphism $\eta\circ d_\eta: V\to\CC_Y$ takes
value $0$ along $M\sub Y$ and takes positive value on $Y-M$. Let $\Gamma_t$ be the graph of $t{\bar\zeta}\in \Gamma(V)$;
and let $\Gamma_\infty$ be the limit of $\Gamma_t$. Then $\Gamma_\infty$ is the normal cone to
$X\sub Y$, embedded in $V|_X$ via the defining equation $X= ({\bar\zeta}=0)\sub Y$. 

Because $\eta\circ{\bar\zeta}=0$, for any $t\in [0,\infty]$,
we have $\eta|_{\Gamma_t}\equiv0$. Thus $(d_\eta\cap \Gamma_t)\cap V|_{Y-M}=\emptyset$.
We then pick a perturbation $d_\eta'$ of $d_\eta$ so that $d_\eta'$ is transversal to both the zero-section
$0_V\sub V$ and $\Gamma_\infty\sub V$, and that $d_\eta'= d_\eta$ away from a  compact subset of 
$Y$. Therefore,
$$e\loc(V,d_\eta)=\zeta\lsta[0_V\cap d_\eta']=\zeta\lsta[\Gamma_\infty\cap d_\eta']=\imath\lsta[X]\virt_{\sigma_\eta}
\in H\lsta(M,\QQ),
$$
where the first identity follows from the definition of the localized Euler classes, the second identity follows from
that $\Gamma_t$ is a homotopy between $[0_V]$ and $[\Gamma_\infty]$ and the union of $d'_\eta\cap \Gamma_t$
for $t\in [0,\infty]$ is contained in a compact subset of $V$, and  the last identity follows from the analytic construction of
the cosection localized virtual cycles  \cite[Appendix]{KL}.
\end{proof}

Recall that 
$$[\barM_{g,\gamma}(G)^p]\virt= 0_{\sigma,\text{loc}}^!\bl [\bC_{X/M}]\br\in A\lsta (M)
$$ 
is constructed via the localized Gysin map $0_{\sigma,\text{loc}}^!$
(cf. Definition-Proposition \ref{defi}) 
applied to 
the virtual normal cone cycle $[\bC_{X/M}]\in Z\lsta (h^1/h^0(E_{X/M}\bul))$ of the relative perfect obstruction
theory $\phi_{X/M}$ (cf. \eqref{XY}).
By the discussion before Lemma \ref{inde}, we have
$h^1/h^0(E\bul_{X/M})\cong h^1/h^0(E\bul_{\zeta})$ and $H^1(E_{X/M}\bul)\cong H^1(E_\zeta\bul)$.

\begin{lemm} \label{lem6.8}
Let the notation be as stated, then \\
(1). the two cosections are equal, i.e.,  
$\sigma_\eta=\sigma: H^1(E\bul_\zeta)\cong H^1(E_{X/M}\bul)\to\sO_X$.\\
(2). 
$[\bC_{X/M}]=[\bC_\zeta]\in Z\lsta( h^1/h^0(E_{X/M}\bul))\cong Z\lsta (h^1/h^0(E\bul_\zeta))$.
\end{lemm}

\begin{proof}[Proof of Theorem \ref{cop-FJRW}]
By Lemma \ref{lem6.6}, we have
$[W_g(\gamma)]\virt_{FJRW}=\imath\lsta [X]\virt_{\sigma_\eta}$. By definitions,
$[X]\virt_{\sigma_\eta}=0^!_{\sigma_\eta,\text{loc}}[\bC_\zeta]$ and $\barM_{g,\gamma}(G)^p)\virt=
0^!_{\sigma,\text{loc}}[\bC_{X/M}]$. Applying Lemma \ref{lem6.8}, we conclude $0^!_{\sigma_\eta,\text{loc}}[\bC_\zeta]
=0^!_{\sigma,\text{loc}}[\bC_{X/M}]\in A\lsta (M)$. This proves the Theorem.
\end{proof}

\begin{proof}[Proof of Lemma \ref{lem6.8}]
We prove (1) first.
Let $S$ be an affine scheme, let
$\rho:S\to X$ and $\rho'=q\circ \rho:S\to M$ be morphisms. As in the discussion leading to
\eqref{etaa}, we assume that $\rho$ is given by 
$s=(s_1,\cdots,s_n)\in \oplus_{j=1}^n \Gamma(\rho^{\prime\ast}\sF^0_j)$.
Recall that $\rho\sta(\eta): \rho\sta V \to \sO_S$ (cf. \eqref{etaa}) sends 
$(\dot s_1,\cdots,\dot s_n)\in \rho\sta V=\oplus_{j=1}^n \rho^{\prime\ast}\sF^1_j$ to
$$\sum_a\alpha_a\rho^{\prime\ast}(\nu_a^1) (W_a(s)_1 \dot s_1,\cdots, W_a(s)_n \dot s_n).
$$
Let $\dot t_j\in H^1(\rho^{\prime\ast}\sF_j\bul)$ be the image of $\dot s_j$.
Then 
$$(\dot t_1,\cdots, \dot t_n)\in \rho\sta\Ob_{X/M}=\oplus _{j=1}^n H^1(\rho^{\prime\ast}\sF\bul_j),
$$
and because $\nu_a$ is the complex realization of the isomorphism $\tau_a$ in \eqref{taua},
and because $\rho\sta(\sigma_\eta)$ is the descent of $\rho\sta(\eta)$, we have
$$\rho\sta(\sigma_\eta)(\dot t_1,\cdots, \dot t_n)=\rho\sta(\eta)(\dot s_1,\cdots, \dot s_n)=
\sum_a\alpha_a\rho^{\prime\ast}(\nu_a^1)(W_a(s)_1 \dot s_1,\cdots, W_a(s)_n \dot s_n)
$$
$$\qquad\qquad\qquad=\sum_a\sum_j \alpha_a\rho^{\prime\ast}(\tau_a)(W_a(s)_j \dot t_j) 
\in \Gamma(\rho^{\prime\ast} R^1\pi_{M\ast}\omega_{\cC_M/M})
$$
where the last identity follows from the same reasoning before Lemma \ref{tri}. 
The last term above is of the same form as (\ref{sigma}), thus we conclude that $\rho\sta(\sigma_\eta)=\rho\sta(\sigma)$. This proves (1).

We next prove (2). 
Aplying Lemma \ref{inde}, it suffices to prove (2) for one complex of locally free
sheaves $F\bul=[F^0\to F^1]$ such that $F\bul\cong \Rpi_{M\ast}\cE$. We choose such a complex $F\bul$ now.

Since $\pi_M: \cC_M\to M$ is a family of pointed stable twisted curves,
$\omega_{\cC_M/M}^{log}$ is $\pi_M$-ample, thus for sufficiently large and divisible $r$, $\cK=(\omega_{\cC_M/M}^{log})^{\otimes r}$ 
is sufficiently ample and is a pullback of an invertible sheaf on the coarse moduli space of $\cC_M$.
Thus $\pi_M\sta\pi_{M\ast} \cK\to \cK$ is surjective. By dualizing this homomorphism, 
taking the cokernel, and then tensoring with $\cK$, we obtain
an exact sequence of locally free sheaves of $\sO_{\cC_M}$-modules
$$0\lra \sO_{\cC_M} \lra (\pi_M\sta\pi_{M\ast} \cK)\dual\otimes \cK\lra \cK_1 \lra 0.
$$
Tensoring with $\cE$, and renaming the second and the third terms to be $\cE^0$ and $\cE^1$, 
we obtain the following exact sequence of locally free sheaves of $\sO_{\cC_M}$-modules
\beq\label{ee}
0\lra \cE\lra \cE^0\lra  \cE^1 \lra 0. 
\eeq
Because $\cK$ is sufficiently ample, we have $R^1\pi_{M\ast} \cE^1=R^1\pi_{M\ast} \cE^0=0$.
We let our new complex $F\bul$ be
$$F\bul=[\zeta: F^0\to F^1]=[\pi_{M\ast}\cE^0\to \pi_{M\ast}\cE^1].
$$
By the vanishing, we have isomorphism of derived objects $F\bul\cong \Rpi_{M\ast}\cE$.

We let $\ti q: Y\to M$, $X=(\bar\zeta=0)\sub Y$, and the obstruction theory $\phi_\zeta$ be constructed from this 
$F\bul$ as before. We prove that $\phi_\zeta=\phi_{X/M}^{\ge -1}$. Note that this implies that $[\bC_\zeta]=
[\bC_{X/M}]$ as cycles in $Z\lsta (h^1/h^0(E_{X/M}\bul))$.

We let ${\cC_Y}=Y\times_M\cC_M$ and ${\cC_X}=X\times_M\cC_M$.  We denote by
$E$, $E^0$ and $E^1$ to be the (total spaces of the) vector bundles associated to $\cE$, $\cE^0$ and $\cE^1$, respectively.
Because $Y$ is the total space of $F^0$, via the projection $\ti q: Y\to M$, 
the isomorphism $F^0=\pi_{M\ast}\cE^0$ induces a tautological section in $ \Gamma(\ti q\sta\pi_{M\ast}\cE^0)$,
which composed with the tautological map $\pi_{M}\sta\pi_{M\ast}\cE^0\to\cE^0$ induces a section 
$e\in \Gamma(\cC_Y,\ti p\sta \cE^0)$, where $\ti p: \cC_Y\to \cC_M$ is the projection.
We let 
$e_Y: \cC_Y\to E^0$ be the evaluation morphism associated with $e$.
Via \eqref{ee}, its restriction to $\cC_X\sub \cC_Y$ then lifts to a section $e_X:{\cC_X}\to E$,
fitting into the Cartesian square
$$\begin{CD}
{\cC_X}@>e_X>>E\\
@VVV@VVV\\
{\cC_Y}@>e_Y>> E^0.
\end{CD}
$$

Let $I_{\cC_X}$ (resp. $I_E$) be the ideal sheaf of ${\cC_X}\sub {\cC_Y}$ (resp. $E\sub E^0$). 
As ${\cC_Y}$ and $E^0$ are smooth over $\cC_M$, we obtain the induced homomorphism
\beq\label{dia}
\begin{CD}
e_X\sta \tLL_{E/\cC_M} @= [e_X\sta\, I_E/I^2_E\to e_X\sta\,\Omega_{E^0/\cC_M}]\\
@VVV @VV{(h_{-1}, h_0)}V\\
\tLL_{{\cC_X}/\cC_M} @= [I_{\cC_X}/I^2_{\cC_X}\to\Omega_{{\cC_Y}/\cC_M}|_{{\cC_X}}]\\
\end{CD}
\eeq 
Let  $\tau:E\to \cC_M$ be the projection. 
Then $I_{E}/I_{E}^2= N\dual_{E/E^0}= \tau\sta \cE^{1\vee}$ and $\Omega_{E^0/\cC_M}|_E= \tau\sta \cE^{0\vee}$.  
Let $f=\tau\circ e_X:{\cC_X}\to \cC_M$ be the projection, thus $e_X\sta\tau\sta=f\sta$, and
$$e_X\sta\, I_E/I^2_E= f\sta \cE^{1\vee} \and  e_X\sta\,\Omega_{E^0/\cC_M} =f\sta \cE^{0\vee}.$$
Let $I_X$ be the ideal sheaf of $X\sub Y$. Because $\pi_X:\cC_X\to X$ is flat,
$I_{\cC_X}/I^2_{{\cC_X}}=\pi_X\sta(I_X/I_X^2)$ and $\tLL_{{\cC_X}/\cC_M}\cong \pi_X\sta\tLL_{X/M}$. 
Thus (\ref{dia}) is identical to
\beq\label{dia1}
\begin{CD}
e_X\sta \tLL_{E/\cC_M} @= [f\sta\cE^{1\vee}\to f\sta \cE^{0\vee}]\\
@VV{h}V@VV{(h_{-1}, h_0)}V\\
\pi_X\sta\tLL_{X/M} @= [\pi_{X}\sta (I_X/I^2_X)\to \pi_{X}\sta \Omega_{Y/M}] .
\end{CD}
\eeq


Applying  the Grothendieck duality 
$$ \varphi: \Hom_{\cC_X}(G_1\dual,\pi_X\sta G_2)\mapright{\cong} \Hom_X((\Rpi_{X\ast}G_1)\dual,G_2),
$$
for $G_1\in D^b(\cC_X)$ and $G_2\in D^b(X)$, to the columns of  (\ref{dia1}), 
and using  $R^1\pi_{X\ast}f\sta \cE^i=0$ becasue of $R^1\pi_{M\ast}\cE^i=0$, $i=0,1$,
we obtain the arrow $\varphi(h)=(\varphi(h_{-1}),\varphi(h_0))$ shown below:
\beq\label{fi}
 \begin{CD}
(F\bul)\dual @= [ (\pi_{X\ast}f\sta \cE^1)\dual\to (\pi_{X\sta}f\sta \cE^0)\dual]\\ 
 @VV{\varphi(h)}V@VV{(\varphi(h_{-1}), \varphi(h_0))}V \\ 
\tLL_{X/M} @= [ I_X/I^2_X\to \Omega_{Y/M}|_X]  .
\end{CD}
\eeq
Following \cite[Prop 2.5]{CL}, we have  $\phi_{X/M}=\varphi(h)\dual$. 
Also $\varphi(h_0)$ is given by $ (\pi_{X\sta}f\sta \cE^0)\dual= (q\sta \cE^0)\dual=  \Omega_{Y/M}|_X $. 
On the other hand, via $ \pi_{X\ast}f\sta \cE^1= q\sta \pi_{M\ast} \cE^1=q\sta F^1=V|_X$,
one checks that $\varphi(h_{-1})=\bar\zeta$.   Thus (\ref{fi}) is identical to (\ref{XE}).
This proves that $(\phi^{\ge -1}_{X/M})\dual=\phi_{\zeta}\dual$, thus $\phi^{\ge -1}_{X/M}=\phi_{\zeta}$.
This proves the Lemma.
\end{proof}

\subsection{Comparison with Polishchuk-Vaintrob's construction}
In this subsection we prove the equivalence with Polishchuk-Vaintrob's construction \cite{PV1} of Witten's top Chern classes.

We first introduce the initial data (triple) $(V,\zeta,\eta)$ over which the Polishchuk-Vaintrob's construction applies.
The triple consists of a pure rank vector bundle $V$ over a pure dimensional DM stack $Y$  over $\CC$, and
sections $\zeta\in\Gamma(Y,V)$ and $\eta\in\Gamma(Y,V\dual)$ such that $\eta\circ\zeta=0$. 
Note that the $(Y,\zeta,\eta)$ constructed in the previous subsection satisfies this requirement.

We let $X=(\zeta=0)\sub Y$, and let $D=(\eta=0)\cap X\sub X$. 
We first recall the prior construction that
gives a cosection localized virtual cycle $[X]\virt\loc$.
Let $\bC_{X/Y}$ be the normal cone of $X$ relative to $Y$, which is a subcone of $V\dual|_X$. We claim that
$\eta|_{\bC_{X/Y}}=0$. Indeed, since $\eta\circ\zeta=0$, and since $\eta: V\to \sO_Y$ is a homomorphism,
$\eta\circ (t\zeta)=0$. As $\bC_{X/Y}$ is the specialization of the graph of $t\zeta$ for $t\to \infty$, we have
$\eta|_{\bC_{X/Y}}=0$.

%
%
We define $\sigma=\eta|_{X}:V|_{X}\to \sO_{X}$; its degeneracy locus (non-surjective locus) is
$X\times_Y (\eta=0)=D$. 
We let $V|_X(\sigma)=V|_D\cup \ker\{V|_{X-D}\to\sO_{X-D}\}$ be the kernel stack of $\sigma$. Then
$[\bC_{X/Y}]\in Z\lsta (V|_X(\sigma))$.
Applying cosection localized  Gysin map $0_{\sigma,loc}^!: A\lsta (V|_X(\sigma))\to A\lsta (D)$, we
obtain the class
$$[X]\virt_{\sigma,loc}\defeq 0^!_{\sigma,loc}[\bC_{X/Y}]\in A\lsta (D).
$$

We remark that when $(V,\zeta,\eta)$ equals to 
$(E_0,V,\zeta,\eta)$ in the previous subsection, one has $X=\barM_{g,\gamma}(G)^p$
and $0^!_{\sigma,loc}[\bC_{X/Y}]=[\barM_{g,\gamma}(G)^p]\virt$, the Witten's top Chern class constructed using
cosection localization.

In \cite{PV1}, Polishchuk-Vaintrob used 
MacPherson's graph construction applied to a double-periodic complex to construct a cycle, 
which applied to the case $(Y,V,\zeta,\eta)=(E_0,Y,\zeta,\eta)$ yields their construction of
Witten's top Chern class of the LG space $([\CC^n/G],W)$. We briefly recall their
construction.

For $(Y,V,\zeta,\eta)$, one forms the vector bundles over $Y$:
$$
S^+=\bigoplus_{k \ \text{even}}\wedge^{k} V\dual,\quad S^-=\bigoplus_{k\ \text{odd}}\wedge^{k} V\dual
\and S=S^+\oplus S^-.
$$
The section $s=(\zeta,\eta)\in \Gamma(Y,V\oplus V\dual)$ defines a two periodic complex
\beq\label{S}S\bul=[\cdots\mapright{\wedge s} S^+\mapright{\wedge s} S^-\mapright{\wedge s} S^+\mapright{\wedge s} 
S^-\mapright{\wedge s}\cdots].
\eeq
Clearly, $S\bul$ is exact outside $D=(s=0)$.

By adopting MacPherson's graph construction to this case, 
they constructed a localized Chern character $ch_{D}^Y(S\bul)\in A\sta(D\to Y)$ of the two-periodic complex $S\bul$, 
and proved that (\cite[Thm 3.2]{PV1})
\beq\label{eq1}td(V|_{D})\cdot ch_{D}^Y(S\bul)\in A^r (D\to Y),\quad
r=\rank V.
\eeq

\begin{defi}[Polishchuk-Vaintrob \cite{PV1}]\label{def-PV}
Let $(V,\zeta,\eta)$ be as stated. Then the Witten's top Chern class of $(V,\zeta,\eta)$ is
\beq\label{ctop}c_{top}(V,\zeta,\eta):=td(V|_{D})\cdot ch_{D}^Y(S\bul)\cdot [Y]\in A_{\ast}(D).
\eeq
\end{defi}

The Witten's top Chern class constructed by Polishchuk-Vaintrob is by applying this 
Definition to the $(E_0,V,\zeta,\eta)$ constructed in the previous subsection.

\begin{prop}\label{P5.10}
Let the notation be as in Definition \ref{def-PV}. Suppose $Y$ is smooth. Then 
$$[X]_{\sigma,loc}\virt=c_{top}(V,\zeta,\eta)\in A\lsta (D).
$$
\end{prop}
  
\begin{coro}
Let $(V,\zeta,\eta)$ be $(E_0,V,\zeta,\eta)$ constructed in the previous subsection.
Then $D=\barM_{g,\gamma}(G)$, $X=\barM_{g,\gamma}(G)^p$, and
$$[\barM_{g,\gamma}(G)^p)]\virt=c_{top}(V,\zeta,\eta)\in A_\ast( \barM_{g,\gamma}(G)).
$$
\end{coro}

We divide the proof of the Proposition into  several Lemmas.
 

\begin{lemm}\label{step1} Let the situation be as in Definition \ref{def-PV}. Suppose $\eta=0$, then 
$$td(V|_{D})\cdot ch_{D}^Y(S\bul)=\zeta\sta [\iota_Y]\in A^{r}(X\to  Y),\quad r=\rank V,
$$
where $\iota_Y:Y\to V$ is the zero section, which defines the class $[\iota_Y]\in A^{r}(Y\to V)$, and
$X\to  Y$ is the inclusion. Consequently, $c_{top}(V,\zeta,0)=[X]_{\sigma,loc}\virt$.
\end{lemm}

\begin{proof} Let $p: V\to Y$ be the projection and $e\in\Gamma(V,p\sta V)$ be the tautological section. 
Let $\wedge^{-\bullet}p\sta V\dual$ be the Koszul  complex (concentrated in degrees $[-r,0]$) induced by $e$. 
Applying \cite[Prop 2.3 (vi)]{PV1}, one obtains
\beq\label{12}ch_{Y}^V(\wedge^{-\bullet}p\sta V\dual)=td(V)^{-1}[\iota_Y].\eeq
Applying  $\zeta\sta$ to (\ref{12}) and using compatibility between pullback and product of invariant classes, one has
\beq\label{ab} ch_{D}^Y(\zeta\sta\wedge^{-\bullet}p\sta V\dual)=\zeta\sta ch_{Y}^V(\wedge^{-\bullet}p\sta V\dual)=td(V|_{D})^{-1}\zeta\sta[\iota_Y]\eeq
where the first identity is by  \cite[Prop 2.3(iii)]{PV1}. Observe that the complex $\zeta\sta\wedge^{-\bullet}p\sta V\dual$ is identical to $[\cdots\to \wedge^k V\dual\mapright{\zeta} \wedge^{k-1} V\dual\to\cdots]$, and also its associated two periodic complex following \cite[Prop 2.2]{PV1} is identical to
(\ref{S}) because $\eta=0$. Hence  \cite[Prop 2.2]{PV1} implies the first term in (\ref{ab}) is equal to $ch_{D}^Y(S\bul)$. This proves the Lemma.
\end{proof}
 
\begin{coro}\label{2.2} If $Z\sub Y$ is a Cartier divisor, and let $U\bul=[\sO_Y\toright{1}\sO_Y(Z)]$, of amplitude $[0,1]$. 
Denote  $ch_Z^Y(U^{\bullet \vee})\in A\sta(Z\to Y)$ to be the localized Chern character defined in \cite[Sect 18.1]{Fu}. Then
$$td[\sO_Y(Z)]\cdot ch_{Z}^Y(U^{\bullet \vee})=[\iota_Z]\in A^1(Z\mapright{\iota_Z} Y)$$
and
$$td[\sO_Y(-Z)]\cdot ch_Z^Y(U^{\bullet })=-[\iota_Z]\in A^1(Z\mapright{\iota_Z} Y).$$
\end{coro}

\begin{proof} Applying Lemma \ref{step1} to  $V=\sO(Z)$, $\zeta=1$ and $\eta=0$, one obtains the first identity. 
The second identity follows from \cite[Example 18.1.2 (b)]{Fu} and the first identity.
\end{proof}
 
\begin{lemm}\label{id}  Let the notation be as stated, and suppose $Y$ is smooth, Then 
$$c_{top}(V,\zeta,\eta)=[X]_{\sigma,loc}\virt.
$$
\end{lemm}

\begin{proof} 
We let $\lam:\ti Y\to Y$ be the blowing up of $Y$ along $(\eta=0)$, $\ti Y$. 
 Let $(\ti V,\ti\zeta,\ti\eta)\defeq (\lam\sta V, \lam\sta \zeta,\lam\sta\eta)$ be the pullback of $(V,\zeta,\eta)$ by $\lam$.
Let $\ti X:=(\ti\zeta=0)=\lam^{-1}(X)$, $\ti D:=(\ti \zeta=0)\cap (\ti\eta=0)=\lam^{-1}(D)$, and
$\ti\sigma=\ti\eta|_{\ti X}$. 

We denote $\lam'=\lam|_{\ti D}:\ti D\to D$. Since $Y$ is smooth, using deformation to normal cone construction, we conclude that
the cycle $\bC_{X/Y}\in Z\lsta (V)$ is the push-forward of $\bC_{\ti X/\ti Y}\in Z\lsta (\ti V)$ under the proper morphism $\ti V\to V$. 
Because Gysin map commutes with proper pushforward, we conclude
$$ \lam'_{\ast}(0^!_{\ti\sigma,loc}[\bC_{\ti X/\ti Y}])=0^!_{\sigma,loc}[\bC_{X/Y}]\in A\lsta (V).
$$ 
 
On the other hand, \cite[Prop 2.3(iii)]{PV1} implies $\lam'_{\ast}(c_{top}(\ti V,\ti\zeta,\ti\eta))=c_{top}(V,\zeta,\eta)$.
Thus to prove the Lemma it suffices to prove the case when (the triple $(V,\zeta,\eta)$) $\eta: V\to \sO_Y$ factors through
a surjective $\eta': V\to \sO_Y(-Z)\sub \sO_Y$ for a Cartier divisor $Z\sub Y$ while $Y$ is not necessary smooth.
Let $K\sub V$ be the kernel of $\eta'$, which is a vector bundle (locally free) since $\eta'$ is surjective. 
Since $\eta(\zeta)=0$, $\zeta\in \Gamma(V)$ lifts to $\zeta'\in\Gamma(K)$. 
 
The exact sequence 
\beq\label{KK}
0\lra K\lra V\lra \sO_Y(-Z)\lra 0
\eeq
usually is non-split. We will use a standard argument to reduce to the split case. 
 Let $\bar Y:=Y\times \Ao$ with projections $\pi_Y:\bar Y\to Y$. We form the vector bundle $\bar V$ on $\bar Y$ that fits
into the exact sequence
$$0\lra \pi_Y\sta K\lra \bar V\lra \pi_Y\sta \sO_Y(-Z)\lra 0
$$
so that its restriction to $t\ne 0$ (resp. $0$) is the exact sequence \eqref{KK} (resp. splits).

We let $\bar\zeta=\pi\sta_Y\zeta'\in \Gamma(\pi_Y\sta K)$.  We also view it as a section in $\bar V$ via the
inclusion $\pi_Y\sta K\sub \bar V$. We let $\bar\eta: \bar V\to \sO_{\bar Y}$ be the composite of
$\bar V\to \pi_Y\sta \sO_Y(-Z)$ with the inclusion $\pi_Y\sta\sO_Y(-Z)\sub \sO_{\bar Y}$.
 
The triple $(\bar V,\bar\zeta,\bar\eta)$ over $\bar Y$ has the property that $\bar X=(\bar\zeta=0)=X\times\Ao$
and $\oD:=\bar X\cap (\bar\eta=0)=D\times \Ao$.  
Let $\iota_t:D=\bar D\times_{\Ao} t\sub \oD$ be the inclusion, and denote $\bar Y_y=\bar Y\times_{\Ao}t$.
Then \cite[Prop 2.3(iii)]{PV1} and \cite[Def 17.1 (C3)]{Fu} imply that
$$\iota_t^! c_{top}(\bar V,\bar\zeta,\bar \eta)=c_{top}(\bar V|_{\bar Y_t},\bar\zeta|_{\bar Y_t},\bar\eta|_{\bar Y_t})\in A\lsta (\bar D_t)=A\lsta (D)
$$ 
is constant in $t\in \Ao$. 

On the other hand, let $\bar X_t=\bar X\times_{\Ao}t$,
and let $[\bar X_t]\virt_{loc}$ and 
$[\bar X]\virt_{loc}$ be the cosection localized virtual classes of the datum $(\bar V_t,\bar\zeta_t,\bar\eta_t)$ and $(\bar V,\bar\zeta,\bar\eta)$.
Since $\bar X$ is a constant family over $\Ao$, \cite[Thm 5.2]{KL} implies that $\iota_t^! [\bar X]\virt_{loc}=[\bar X_t]\virt_{loc}\in A\lsta (D)$ 
is constant in $t$. Thus to prove the Lemma we only need to show that $\iota_0^! c_{top}(\bar V,\bar\zeta,\bar\eta)=[\bar X]\virt_{loc}$. 
Equivalently, we only need to prove that in case \eqref{KK} splits, $c_{top}(V,\zeta,\eta)=[X]_{\sigma,loc}\virt$.

 
Using that $X=(\zeta=0)$ and $\zeta$ lifts to a section in $K\sub V$, the normal cone $\bC_{X/Y}\sub K|_{X}$, and $0_K^![\bC_{X/Y}]\in A\lsta(X)$. 
Following the definition of cosection localization \cite{KL}, we conclude
\beq\label{xa}[X]_{\sigma,loc}\virt=-\iota_{Z}^!0^![\bC_{X/Y}]\in A\lsta(D)
\eeq
where $\iota_{Z}^!:A\lsta (X)\to A_{\ast-1}(D)$ is the Gysin map associated to the Cartesian square
$$
\begin{CD}
D@>>> X\\
@VVV@VVV\\
Z@>{\iota_{Z}}>>Y.
\end{CD}
$$

 
Next we analyze $c_{top}(V,\zeta,\eta)$. The Koszul complex
$$C\bul:=[\cdots \to \wedge^k K\dual\mapright{\zeta}\wedge^{k-1} K\dual\to\cdots]
$$
(with $-k$ to be the amplitude of $\wedge^k K\dual$) is of finite length. Also let 
$$U\bul:=[\sO_Y\mapright{\sub} \sO_Y(Z)]
$$ 
be of amplitude $[0,1]$. As $C\bul$ is exact outside $X$ and $U\bul$ is exact outside $Z$, their tensor product
$$\fV\bul:=C\bul\otimes U\bul=[\cdots\to \fV^{k}\mapright{d_k}\fV^{k-1}\to\cdots]
$$ 
is exact outside $X\times_{Y}Z=D$ and of finite length. This complex leads to  a complex 
$$S_{\fV}\bul=[\cdots\to S_\fV^+\mapright{d^+} S_\fV^-\mapright{d^-} S_\fV^+\to\cdots].
$$
where $S_\fV^+=\oplus_k \fV^{2k}$ and $S_\fV^-=\oplus_k\wedge \fV^{2k+1}$, with differentials $d^+=\oplus_k d_{2k}$ and 
$d^-=\oplus_k d_{2k+1}$.

We next introduce
$$J_+=\oplus_k \wedge^{2k} V\dual \and J_-=\oplus_k \wedge^{2k-1} V\dual,
$$
and let $\zeta_\pm:J_\pm\to J_\mp$ be induced by contracting with $\zeta$. We have
$$S_\fV^+=J_+\oplus J_-(Z) \and S_\fV^-=J_-  \oplus J_+(Z).
$$ 
For $a\in J_+,a'\in J_-(Z)$ and $b\in J_-,b'\in J_+(Z)$, the boundary maps take form
$$d_+(a,a')=(\zeta_+(a),a\otimes1_Z+\zeta_-(a')) \and d_-(b,b')=(\zeta_-(b),b\otimes 1_Z+\zeta'(b')),
$$
where $1_Z$ is the section of $\sO_Y(Z)$ as the image of $1$ under $\sO_Y\to \sO_Y(Z)$.
Using the splitting $V=K\oplus \sO_Y(-Z)$ and its compatibility with $\zeta,\eta$, one checks 
that the two-periodic complex $S\bul$ is identical to $S_\fV\bul$. Thus   
$$ch_{D}^Y(S\bul)=ch_{D}^Y(S_\fV\bul)=ch_{D}^Y(\fV\bul),
$$
where the last term is the localized Chern character defined in \cite[Sect. 18]{PV1} (for the complex $\fV\bul$) 
and the last  equality is by \cite[Prop 2.2]{PV1}. Now as $\fV\bul=C\bul\otimes U\bul$, \cite[Example 18.1.5]{Fu} implies 
$$ch_{D}^Y(\fV\bul)= ch_{Z}^Y(U\bul)\cup ch_{X}^Y(C\bul).
$$
Using the splitting $V=K\oplus \sO_Y(-Z)$, and letting $\iota_X:X\to Y$  be the inclusion, (\ref{ctop}) becomes
\begin{eqnarray*} c_{top}(Y,V,\zeta,\eta)&:=&td(V|_{D})\cdot ch_{D}^Y(S\bul)\cdot [Y] \\ 
   &= &  [td(\sO_Y(-Z)|_Z)\cdot  ch_{Z}^Y(U\bul)]\cup [td(K|_{X})\cdot ch_{X}^Y(C\bul)]\cdot [Y]\\
   &=& \iota_X\sta([td(\sO_Y(-Z)|_Z)\cdot  ch_{Z}^Y(U\bul)]) (0^![\bC_{X/Y}])\\
   &=&-\iota_{Z}^!0^![\bC_{X/Y}]=[X]_{\sigma,loc}\virt,
\end{eqnarray*}
where the second equality follows from \cite[Prop 17.8.2]{Fu}, the third equality follows from Lemma \ref{step1} for $(Y,K,\zeta,0)$, 
and the last equality follows from Corollary \ref{2.2}.
\end{proof}

Combined, we have proved Proposition \ref{P5.10}.

%
%
%
%
%
%

\end{document}